\documentclass{article}
\usepackage[utf8]{inputenc}
\usepackage[affil-it]{authblk}
\usepackage{amsmath}
\usepackage{amsbsy}
\usepackage{amssymb}
\usepackage{graphicx}
\usepackage{dsfont}
\usepackage{upgreek}
\usepackage{textcomp}
\usepackage{braket}
\usepackage{setspace}
\usepackage{blindtext} 
\usepackage{blindtext} 
\usepackage[margin=1in]{geometry}
\usepackage{hyperref}
\hypersetup{
    colorlinks=true, 
    linktoc=all,     
    linkcolor=blue,  
}
\usepackage{tocloft} 
\usepackage{amsthm}
\usepackage{mathrsfs}
\usepackage{mathtools}
\usepackage[table,svgnames]{xcolor}
\usepackage{graphicx}
\usepackage{tikz}
\usepackage{float}
\usepackage{enumerate}
\usetikzlibrary{arrows}
\usepackage[toc,page]{appendix}
\usepackage{cleveref}
\usepackage{xcolor}

\Crefname{appsec}{appendix}{appendices}
\numberwithin{equation}{section}

\newtheorem{theorem}{Theorem}[section]
\newtheorem{lemma}{Lemma}[section]
\newtheorem{corollary}{Corollary}[section]
\newtheorem{remark}{Remark}[section]

\newtheorem{proposition}{Proposition}[section]

\providecommand{\keywords}[1]
{
  \textbf{Key Words and Phrases.} #1
}
\providecommand{\classification}[1]
{
  \textbf{Mathematics Subject Classification.} #1
}
\title{Continuum Limit of 2D Fractional Nonlinear Schr\"odinger Equation}
\author{Brian Choi\thanks{Corresponding author. Contact: \texttt{choigh@smu.edu}}, Alejandro Aceves\thanks{Contact: \texttt{aaceves@smu.edu}}%
  }
\affil{Department of Mathematics, Southern Methodist University, Dallas, TX 75275, USA}
\date{}
\begin{document}
\maketitle\vspace{-8ex}
\maketitle
\begin{abstract}
We prove that the solutions to the discrete Nonlinear Schr\"odinger Equation (DNLSE) with non-local algebraically-decaying coupling converge strongly in $L^2(\mathbb{R}^2)$ to those of the continuum fractional Nonlinear Schr\"odinger Equation (FNLSE), as the discretization parameter tends to zero. The proof relies on sharp dispersive estimates that yield the Strichartz estimates that are uniform in the discretization parameter. An explicit computation of the leading term of the oscillatory integral asymptotics is used to show that the best constants of a family of dispersive estimates blow up as the non-locality parameter $\alpha \in (1,2)$ approaches the boundaries.
\end{abstract}
\classification{35B30, 35Q40, 35Q55, 35Q60, 35R11, 37K60}\\
\section{Introduction.}

The mathematical description of physical phenomena, in many instances, results in the formulation of partial differential equations (PDEs) describing  state variables in continuum media. Despite the fact that  it is highly unlikely to find exact solutions of many linear or nonlinear PDEs, advances in numerical analysis and scientific computing open the door to find approximate solutions to complex problems. In particular, numerical approximations based on finite difference schemes are constructed by discretizing spatial variables, leading to a system of coupled ordinary differential equations. In this line of research, the objective is then to determine how well the approximate solution evaluated in the grid approximates the solutions of the corresponding PDE. 

On the other hand, there are well known universal models that are {\it inherently discrete}. Generically referred to as coupled oscillator systems, they describe phenomena such as localization or synchronization, characteristic of its discrete nature. Best known examples are the Fermi–Pasta–Ulam–Tsingou model, the discrete nonlinear Schr\"odinger Equation and the Kuramoto model. The first two describe dynamics in a lattice with nearest neighbor interactions, whereas the Kuramoto model addresses synchronization for globally coupled oscillators. These and similar models continue to be studied given their applicability in photonics, lasers, and networks such as the power grid to name some. For such models a suitable approximation named the long-wave approximation assumes a ``smooth" variation of the state variable amongst neighbor lattices. Specifically in a 1-dimensional lattice, this means $u_{n\pm 1} \approx u_n$. In this regime, it is reasonable to consider continuum approximation. For a 1-d lattice model, the continuum approximation $u_{n \pm 1} \rightarrow U(x \pm h)$, where $h>0$ is small, with nearest neighbor coupling $C(u_{n+1} + u_{n-1})$ leads to a term proportional to $\frac{\partial^2U}{\partial x^2}$ and in return, the system of ODEs is then approximated by a PDE.

Recently, there has been an increased interest in the models based on FNLSE. While most of the research deals with continuum models, including numerical computations of solutions in the nonlinear regime, less is known about discrete systems showing global coupling with algebraic decay on the coupling strength with respect to the distance between nodes in the lattice. This work considers such a case in a two-dimensional lattice and centers on the question of the validity of a suitable continuum approximation. This is not always a trivial task as, for instance, invariances and symmetries may arise or be lost. In contrast to the (continuum) nonlinear Schr\"odinger equation that admits the Galilean boost from which traveling wave solutions emerge, many lattice systems lack translational invariance. It is known that highly localized solutions in a lattice system do not propagate due to the presence of the Peirels-Nabarro potential
\cite{Kivshar,hocking2022topological}; for a recent work on FNLSE in this context, see \cite{jenkinson2017}. All this is to point out the challenges and open problems that need to be studied by a combination of analytical and numerical tools.  In this contribution, we report what we think are first analytic results on the underlying fundamental question of determining the continuum approximation on the FNLSE in more than one dimension.

\section{Statement of the problem.}

This work concerns the continuum limit of the discrete fractional nonlinear Schr\"odinger equation (FNLSE)
\begin{equation}\label{maineq1}
i \dot{u}_h = (-\Delta_h)^{\frac{\alpha}{2}}u_h + \mu |u_h|^{p-1}u_h,\:u_h(x,0)=u_{0,h}(x),
\end{equation}
to the continuum FNLS
\begin{equation}\label{maineq2}
i\partial_t u = (-\Delta)^{\frac{\alpha}{2}}u + \mu |u|^{p-1} u,\: u(x,0) = u_0(x),
\end{equation}
as $h\rightarrow 0+$ where $\alpha \in (0,2]\setminus \{1\},\: p>1,\: \mu = \pm 1,$ and $u: \mathbb{R}^{2+1}\rightarrow \mathbb{C},\:u_h:h\mathbb{Z}^2\times\mathbb{R}\rightarrow\mathbb{C}$. Let \eqref{maineq1}, \eqref{maineq2} be well-posed in some Banach spaces $X,X_h$, respectively, where $0<h\leq 1$ denotes a discretization parameter. Suppose $u_{0,h}\in X_h$ is the discretized $u_0\in X$. Given an interpolation operator $p_h:X_h \rightarrow X$ and $T>0$ such that $u(t),u_h(t)$ denote the well-posed solutions on $[0,T]$, the main problem then reduces to identifying values of $\alpha,p$ that allows
\begin{equation*}
\lim_{h\rightarrow 0}\sup_{t\in [0,T]}\| p_h u_h(t) - u(t) \|_{X}=0.
\end{equation*}

The study of evolution equations on $\mathbb{R}$ with a general class of interaction kernel was done in \cite{kirpatrick} where the continuum limit was proved in the weak sense. By applying the analytic tools in \cite{hong2018uniform} that yield dispersive estimates for the discrete Schr\"odinger evolution that are uniform in $h$, \cite{hong2019strong} extended the aforementioned weak convergence to strong convergence in the $L^2$-setting (with convergence rates) for $\alpha=2$ in $\mathbb{R}^d,\:d=1,2,3$ and $\alpha \in (0,2)\setminus\{1\}$ on $\mathbb{R}$. The central perspective in \cite{hong2019strong}, upon which we develop, that sharp dispersive estimates that are uniform in $h$ control the difference $p_h u_h - u$, at least in the scaling-subcritical regime, proved to be fruitful as can be illustrated in various works such as \cite{hong2021finite} that studied the case $\alpha=2$ on $\mathbb{T}^2$ as the spatial domain, \cite{hong2021continuum} that studied the large box limit for $\alpha=2$ in $\mathbb{R}^d,\: d=2,3$, and \cite{hong2021korteweg} that showed the rigorous derivation of the KdV equation from the FPU system. Using a similar idea, the continuum limit of the space-time FNLS was investigated in \cite{grande2019continuum}. Furthermore see the works of Ignat and Zuazua \cite{ignat2005dispersive,ignat2005two,ignat2009numerical,zuazua2009convergence,ignat2012convergence} where novel approaches such as the Fourier filtering and the two-grid algorithm were used.

In practice, obtaining appropriate dispersive estimates reduces to oscillatory integral estimations, which is of central concern in our approach. Unlike the continuum case, the dispersion relation for the discrete evolution has degenerate critical points, which results in weaker dispersion than the continuum Schr\"odinger evolution. This in return admits weaker Strichartz estimates, which limits the class of nonlinearities that leads to the well-posedness of the corresponding nonlinear equation via the contraction mapping argument. To be more quantitative, let $U(t) = e^{-it(-\Delta)^{\frac{\alpha}{2}}},\: U_h(t) = e^{-it(-\Delta_h)^{\frac{\alpha}{2}}}$ and $\| f \|_{L^p_h} := h^{\frac{d}{p}}(\sum\limits_{x\in h\mathbb{Z}^2}|f(x)|^p)^{\frac{1}{p}}$ for $p<\infty$ with $\| f \|_{L^\infty_h} = \| f \|_{L^\infty(h\mathbb{Z}^2)}$; see \Cref{notation} for notations. For $\alpha=2$, \cite[Theorem 1]{stefanov2005asymptotic} established\footnote{Denote $A \lesssim B$ when there exists a constant of non-interest $C>0$ such that $A \leq CB$ and define $A \simeq B$ if $A \lesssim B$ and $B \lesssim A$.} 
\begin{equation*}
\| U_h(t) \|_{L^1_h\rightarrow L^\infty_h} \simeq_h |t|^{-\frac{d}{3}},    
\end{equation*}
where the implicit constant blows up as $h\rightarrow 0+$, which contrasts with $\| U(t) \|_{L^1(\mathbb{R}^d)\rightarrow L^\infty(\mathbb{R}^d)} \simeq |t|^{-\frac{d}{2}}$. Our objective is to obtain Strichartz estimates for the discrete evolution that are uniform in $h$.

For $\alpha<2$, \cite[Proposition 3.1]{hong2019strong} obtained 
\begin{equation}\label{dispersive est2}
\|U_h(t)P_N f \|_{L^\infty_h} \lesssim_\alpha (\frac{N}{h})^{1-\frac{\alpha}{3}}|t|^{-\frac{1}{3}}\| f \|_{L^1_h},\:\alpha \in (1,2)
\end{equation}
for all $N\in 2^\mathbb{Z}$ with $N\leq 1$ on $h\mathbb{Z}$ where the $P_N$ denotes the Littlewood-Paley operator. Our goal is to obtain a two-dimensional analog of \eqref{dispersive est2}. The proof in \cite{hong2019strong} cannot be directly generalized, however, since the set of degenerate critical points on $h\mathbb{Z}$ consists of isolated points whose corresponding oscillatory integrals cannot be estimated directly by the Van der Corput Lemma. In higher dimensions, the set of degeneracy is geometrically more complicated. In fact, our analysis shows that the degenerate critical points define a one-dimensional embedded smooth submanifold in the torus $[-\pi,\pi]^2$ where each singular point admits a unique direction along which the third derivative does not vanish (fold) except at four points (cusp) at which the fourth derivative does not vanish. This observation that a singular point is at worst a cusp is consistent with \cite{borovyk2017klein}. It is expected that more severe singularities exist in higher dimensions as the structure of the Hessian of the dispersion relation becomes more complicated. This dimension-dependent geometric complication is purely a remnant of non-locality since the linear evolution of classical Schr\"odinger operator on $h\mathbb{Z}^d$ splits as the $d$-fold tensor product on each dimension.

Consider the dispersion relation
\begin{equation*}
w_{h,m}(\xi) = \Big(m^2+\frac{4}{h^2}\sum_{i=1}^2 \sin^2(\frac{h\xi_i}{2})\Big)^{\frac{\alpha}{2}},
\end{equation*}
and the quantity of interest
\begin{equation*}
    \int_{\frac{\mathbb{T}^2}{h}} e^{i(x \cdot \xi - t w_{h,m}(\xi))}\eta(\xi)d\xi,
\end{equation*}
where $\mathbb{T} =\frac{\mathbb{R}}{2\pi\mathbb{Z}}= [-\pi,\pi]$ and $\eta \in C^\infty_c(\frac{\mathbb{T}^2}{h})$; the dispersion relation of \eqref{maineq2} is $w_{h,0}$. \cite{schultz1998wave} showed that when $m=0,\:\alpha=1$, which corresponds to the dispersion relation of the discrete wave equation, then the quantity of interest decays as $O(t^{-\frac{2}{3}})$ in $d=2$ and $O(t^{-\frac{7}{6}})$ in $d=3$. When $m>0,\:\alpha=1$, which corresponds to the discrete Klein-Gordon equation, \cite{borovyk2017klein} showed that the quantity of interest decays as $O(t^{-\frac{3}{4}})$ in $d=2$, and the result was extended to higher dimensions ($d=3,4$) in \cite{cuenin2021sharp}. When $m=0,\:\alpha=2$, the time decay of the fundamental solution of the classical discrete Schr\"odinger equation was shown to be $O(t^{-\frac{d}{3}})$ in \cite{stefanov2005asymptotic}.

Our objective is to obtain the sharp time decay of the quantity of interest for $m=0,\:\alpha \in (1,2)$ in $d=2$. In particular, it is shown that the oscillatory integral decays as $O(t^{-\frac{3}{4}})$. The main tool that we adopt is the analysis of Newton polyhedron generated by the Taylor expansion of the phase function $x\cdot \xi - t w_{h,0}(\xi)$ in an adapted coordinate system, a method pioneered in \cite{varchenko1976newton}. Furthermore the asymptotics in both regimes $\alpha\rightarrow 1+$ (wave limit) and $\alpha\rightarrow 2-$ (Schr\"odinger limit) are studied. To our knowledge, the dependence on the non-local parameter has not been clearly investigated in previous works. To obtain the asymptotics of the leading term of $O(t^{-\frac{3}{4}})$ as a function of $\alpha$, we represent the phase function in a superadapted coordinate system to apply results of \cite{greenblatt2009asymptotic}.

The relation of our work to the theory of stability of degenerate oscillatory integrals is subtle. A cursory observation might suggest that a degenerate integral (our quantity of interest) would be stable under a small perturbation in the non-local parameter. However the phase fails to be smooth for $\alpha<2$, and therefore becomes large in appropriate norm(s) as the support of $\eta$ becomes arbitrarily close to the origin. In our approach, it suffices to invoke the stability result \cite{ikromov2011uniform} under linear perturbations in phase. For more general stability results under analytic or smooth perturbations, see \cite{karpushkin1986theorem,greenblatt2014stability}. For the support of $\eta$ close to the origin, $\sin z \sim z$ by the small angle approximation, after which one might wish to invoke \cite{cho2011remarks} that obtained sharp dispersive estimates for radial dispersion relations. However such approximation is not a linear perturbation and hence we handle that case by direct computation.

The paper is organized as follows. Notations and main results are presented in \Cref{notation}. Assuming the results hold, the desired continuum limit is shown in \Cref{continuum limit}. The proof of our main proposition is in \Cref{main proof}, followed by a concluding remark in \Cref{conclusion}.

\section{Main results.}\label{notation}
To discuss continuum limit, the parameters that yield the well-posedness of \eqref{maineq1}, \eqref{maineq2} must be identified. For the discrete equation, the linear operator
\begin{equation}\label{fractional}
\Delta_h f (x) = \sum_{i=1}^d \frac{f(x+he_i)+f(x-he_i)-2f(x)}{h^2},\: x \in h\mathbb{Z}^d,
\end{equation}
defines a bounded, non-negative, self-adjoint operator on $L^2_h$, and so are its fractional powers given by functional calculus. Equivalently $(-\Delta_h)^{\frac{\alpha}{2}}$ is given by the Fourier multiplier
\begin{equation}\label{fractional2}
    (-\Delta_h)^{\frac{\alpha}{2}} = \mathcal{F}_h^{-1} \Big\{\sum_{i=1}^d \frac{4}{h^2} \sin^2\Big(\frac{h\xi_i}{2}\Big)\Big\}^{\frac{\alpha}{2}}\mathcal{F}_h, 
\end{equation}
where the discrete Fourier transform is defined as
\begin{equation*}
\begin{split}
\widehat{f}(\xi) = \mathcal{F}_h f (\xi)=h^d \sum_{x \in h\mathbb{Z}^d} f(x)e^{-i x \cdot \xi},\: f(x) = (2\pi)^{-d} \int_{\frac{\mathbb{T}^d}{h}} \widehat{f}(\xi)e^{ix\cdot \xi}d\xi,
\end{split}
\end{equation*}
for $\xi \in \frac{\mathbb{T}^d}{h}$. Recall the Sobolev space on $h\mathbb{Z}^d$ for $s \in \mathbb{R},\: p \in (1,\infty)$ given by
\begin{equation*}
\begin{split}
\| f \|_{W^{s,p}_h} = \| \langle \nabla_h \rangle^s f \|_{L^p_h},\: \| f \|_{\dot{W}^{s,p}_h} = \| | \nabla_h |^s f \|_{L^p_h},
\end{split}
\end{equation*}
where
\begin{equation*}
    \langle \nabla_h \rangle^s = \mathcal{F}_h^{-1} \langle \xi \rangle^s \mathcal{F}_h,\: | \nabla_h |^s = \mathcal{F}_h^{-1} | \xi |^s \mathcal{F}_h,
\end{equation*}
and $\langle \xi \rangle = (1+|\xi|^2)^{\frac{1}{2}}$ for $\xi \in \frac{\mathbb{T}^d}{h}$. The nonlinearity $u_h\mapsto |u_h|^{p-1}u_h$ is locally Lipschitz continuous due to $L^2_h \hookrightarrow L^\infty_h$, which yields an immediate well-posedness of \eqref{maineq1} in $L^2_h$ via the contraction mapping argument. For the continuum case, consider the family of self-similar solutions
\begin{equation*}
u(x,t) \rightarrow u_\lambda(x,t) := \lambda^{-\frac{\alpha}{p-1}}u\Big(\frac{x}{\lambda},\frac{t}{\lambda^\alpha}\Big),
\end{equation*}
and observing that $\{u_\lambda(\cdot,t)\}_{\lambda>0}$ leaves $\dot{H}^{s_c}(\mathbb{R}^d)$ invariant for all $t$, one obtains the Sobolev-critical regularity
\begin{equation*}
s_c = \frac{d}{2} - \frac{\alpha}{p-1}.   
\end{equation*} 
Our analysis is in the scaling-subcritical regime where the time of existence depends on the Sobolev norm of data. Moreover suppose the power of nonlinearity is at least cubic.
\begin{lemma}\label{dnls}
FNLSE \eqref{maineq2} is locally well-posed in $H^s(\mathbb{R}^2)$ for $s>s_c$ and $p \geq 3$ in the subcritical sense. For any $\alpha>0,\:p>1,\:d\in \mathbb{N}$, DNLSE \eqref{maineq1} is globally well-posed in $L^2_h$. Moreover they admit conserved mass and energy functionals given by
\begin{equation*}
\begin{split}
M[u(t)] &= \|u(t) \|^2_{L^2(\mathbb{R}^2)},\: E[u(t)] = \frac{1}{2}\int_{\mathbb{R}^2} ||\nabla|^{\frac{\alpha}{2}}u|^2 dx +\frac{\mu}{p+1} \int_{\mathbb{R}^2} |u|^{p+1}dx,\\
M_h[u_h(t)] &=\| u_{h}(t)\|_{L^2_h}^2,\: E_h[u_h(t)]=\frac{1}{2} \| (-\Delta_h)^{\frac{\alpha}{4}} u_{h}(t) \|_{L^2_h}^2 + \frac{\mu}{p+1}\| u_{h}(t) \|_{L^{p+1}_h}^{p+1}.
\end{split}
\end{equation*}
\end{lemma}
\begin{proof}
See \cite[Theorem 1.1]{1534-0392_2015_6_2265} and \cite[Proposition 4.1]{kirpatrick} for the first and second statement, respectively.
\end{proof}
More specifically, our set-up is in the mass supercritical and energy subcritical regime, or equivalently,
\begin{equation}\label{balance}
\frac{2(p-1)}{p+1}<\alpha<2,
\end{equation}
in which every $u_0 \in H^{\frac{\alpha}{2}}(\mathbb{R}^2)$ has a local solution but not necessarily global; for blow-up criteria in the focusing mass supercritical case via localized virial estimates, see \cite{boulenger2016blowup,dinh2018blow}. 

We specify the discretization described in the introduction. For $h>0$, define $d_h:L^2(\mathbb{R}^d)\rightarrow L^2_h$ by
\begin{equation*}
    d_h f(x) = h^{-d}\int_{x+[0,h)^d} f(x^\prime)dx^\prime.
\end{equation*}
Conversely define $p_h:L^2_h\rightarrow L^2(\mathbb{R}^d)$ by
\begin{equation*}
\begin{split}
p_h f(x) &= f(x^\prime) + D^+_h f (x^\prime) \cdot (x-x^\prime),\: x \in x^\prime + [0,h)^d,\: x^\prime \in h\mathbb{Z}^d\\
(D^+_h f)_i (x^\prime) &= \frac{f(x^\prime+he_i)-f(x^\prime)}{h},\: i=1,\dots,d,
\end{split}
\end{equation*}
where $\{e_i\}_{i=1}^d$ generates $\mathbb{Z}^d$. The discretization converges to the continuum solution. 
\begin{theorem}\label{mainthm}
Let $p \geq 3$ and $\max(\frac{8}{7}, \frac{2(p-1)}{p+1})<\alpha<2$. For any arbitrary $u_0 \in H^{\frac{\alpha}{2}}(\mathbb{R}^2)$, let $u \in C([0,T];H^{\frac{\alpha}{2}}(\mathbb{R}^2)),\: u_h \in C([0,T];L^2_h)$ be the well-posed solutions from \Cref{dnls} where $T = T(\| u_0 \|_{H^{\frac{\alpha}{2}}})>0$. Then there exists $C_i = C_i(\| u_0 \|_{H^{\frac{\alpha}{2}}})>0,\: i=1,2$ independent of $h>0$ such that
\begin{equation}\label{error estimate}
    \| p_h u_h (t) - u(t)\|_{L^2(\mathbb{R}^2)} \leq C_1 h^{\frac{\alpha}{2+\alpha}} ( \| u_0 \|_{H^{\frac{\alpha}{2}}}+\| u_0 \|_{H^{\frac{\alpha}{2}}}^p) e^{C_2 |t|},\: t \in [0,T].
\end{equation}
\end{theorem}
\begin{remark}
To estimate the nonlinear part of $p_hu_h - u$ uniformly in $h$, we show that an appropriate space-time Lebesgue norm of $u_h$ is uniformly bounded in $[0,T(\| u_0 \|_{H^{\frac{\alpha}{2}}})]$ (see \Cref{discrete linfty}). However our proof is insufficient to conclude that a similar uniform bound holds in the energy-critical case, and therefore our method does not extend, at least directly, when $\alpha = \frac{2(p-1)}{p+1}$.
\end{remark}
\begin{remark}
The result is local in time, and thus it is of interest to extend \eqref{error estimate} such that the estimate holds for $t \in [0,T_e)$ where $0<T_e\leq \infty$ is the maximal time of existence of \eqref{maineq2}. This extension is not straightforward due to the existence of finite-time blow up solutions in the mass supercritical regime. For example if $T_e<\infty$, then $\lim\limits_{t\rightarrow T_e -} \| u(t) \|_{H^{\frac{\alpha}{2}}(\mathbb{R}^2)}=\infty$. Since $p_h u_h = O_h(1)$ by \Cref{series}, for all $h>0$ we have
\begin{equation*}
\sup_{t\in [0,T_e)}\| p_h u_h (t) - u(t)\|_{H^{\frac{\alpha}{2}}(\mathbb{R}^2)} = \infty. 
\end{equation*}
\end{remark}
\begin{remark}
Suppose \eqref{maineq2} were discretized by another means. Let $A_h$ be a self-adjoint linear operator on $L^2_h$ and let $v_h\in C([0,T];L^2_h)$ be a solution of
\begin{equation}\label{maineq3}
i \dot{v}_h = A_h v_h + N(v_h),\:v_h(x,0)=u_{0,h}(x).
\end{equation}
Recall that $u_h(t),u(t)$ are well-posed $L^2$-solutions of \eqref{maineq1}, \eqref{maineq2}. If $\| u_h(t)-v_h(t)\|_{L^2_h}\lesssim h^\theta$, then $\| p_h v_h(t) - u(t)\|_{L^2(\mathbb{R}^2)} \lesssim h^{\theta^\prime}$ for some $\theta,\theta^\prime>0$ by \eqref{error estimate} and the triangle inequality.
\end{remark}
It is expected that our approach would apply to a general class of discrete models governed by $\{A_h\}$. A priori, $A_h$ is assumed to act on $L^2_h$ and thus its extension to $L^2(\mathbb{R}^d)$ needs to be defined, after which, the limit of $A_h$ as $h \rightarrow 0$, if it exists, is considered. Let $m_h \in A\left(\frac{\mathbb{T}^d}{h}\right)$ and define $\mathcal{F}_h( A_h f)(\xi) = m_h(\xi)\mathcal{F}_h f (\xi)$ where
\begin{equation*}
A\left(\frac{\mathbb{T}^d}{h}\right) = \{f \in L^\infty\left(\frac{\mathbb{T}^d}{h}\right):\mathcal{F}_h^{-1}f \in L^1_h\}.    
\end{equation*}
Denote $\nu_h = \mathcal{F}_h^{-1}m_h$. Since the Fourier coefficients are absolutely integrable, $\nu_h$ can be interpreted as a complex Borel measure on $\mathbb{R}^d$ given by
\begin{equation*}
    \nu_h(x) = \sum_{y \in h\mathbb{Z}^d} \mathcal{F}_h^{-1}m_h(y) \delta_y(x),
\end{equation*}
where $\delta_y$ is the Dirac mass at $y \in h \mathbb{Z}^d$. Then for $f \in L^2_h,\ x \in h\mathbb{Z}^d$,
\begin{equation*}
    A_h f (x) = \nu_h \ast_h f (x) := h^d \sum_{y\in h\mathbb{Z}^d}\mathcal{F}_h^{-1}m_h(y)f(x-y). 
\end{equation*}
For $f \in C^\infty_c(\mathbb{R}^d)$, we have $f \ast \delta_{y} (x) = \int_{\mathbb{R}^d} f(x-y^\prime) d\delta_{y}(y^\prime) = f(x-y)$, and therefore
\begin{equation}\label{convolution}
    A_h f (x) = h^d \nu_h \ast f(x),
\end{equation}
for $f \in \bigcup\limits_{p \in [1,\infty]} L^p(\mathbb{R}^d)$ and 
\begin{equation}\label{convolution2}
\| A_h f \|_{L^p(\mathbb{R}^d)} \leq h^d \| \nu_h \|_{TV} \| f \|_{L^p(\mathbb{R}^d)},    
\end{equation}
where $\| \nu_h \|_{TV}$ measures the total variation.

\begin{proposition}\label{multiplier}
Define $A_h$ as \eqref{convolution}. Then $A_h:L^p(\mathbb{R}^d)\rightarrow L^p(\mathbb{R}^d)$ is bounded for all $p \in [1,\infty]$ with the operator norms satisfying
\begin{equation*}
    \| A_h \|_{L^p\rightarrow L^p} \geq \| A_h \|_{L^2\rightarrow L^2} = \| m_h \|_{L^\infty(\frac{\mathbb{T}^d}{h})}.
\end{equation*}
\end{proposition}
\begin{proof}
Since $A_h$ is a convolution against a finite measure with bounded symbol $m_h$, $A_h$ is a translation-invariant bounded linear operator on $L^p(\mathbb{R}^d)$ for all $p\in [1,\infty]$ that satisfies \eqref{convolution2}. Since $A_h$ is bounded on $L^p(\mathbb{R}^d)$, it is bounded on $L^{p^\prime}(\mathbb{R}^d)$ by duality. By the Riesz-Thorin Theorem, we have
\begin{equation*}
    \| A_h \|_{L^2 \to L^2} \leq \| A_h \|_{L^p \to L^p}^{\frac{1}{2}} \| A_h \|_{L^{p^\prime} \to L^{p^\prime}}^{\frac{1}{2}} = \| A_h \|_{L^p \to L^p}.
\end{equation*}
The last equality is given by the fact that any translation-invariant bounded linear operator on $L^2(\mathbb{R}^d)$ is given by a bounded multiplier on the Fourier space. 
\end{proof}
As an example, consider two classes of multipliers
\begin{equation}\label{multiplier2}
\sigma_h(\xi) = \Big(\frac{4}{h^2}\sum_{i=1}^d \sin^2(\frac{h\xi_i}{2})\Big)^{\frac{\alpha}{2}},\ m_h(\xi) = c_{d,\alpha} h^d \sum_{z \in h\mathbb{Z}^d\setminus \{0\}} \frac{1-\cos \xi\cdot z}{|z|^{d+\alpha}},
\end{equation}
where $c_{d,\alpha} = \frac{4^{\frac{\alpha}{2}}\Gamma(\frac{d+\alpha}{2})}{\pi^{\frac{d}{2}}|\Gamma(-\frac{\alpha}{2})|}$. It can be verified that $m_h$ defines 
\begin{equation}\label{long-range}
A_h f(x) = c_{d,\alpha}h^d \sum\limits_{y\in h\mathbb{Z}^d \setminus \{x\}}\frac{f(x)-f(y)}{|x-y|^{d+\alpha}}. 
\end{equation}
\begin{proposition}
Both $(-\Delta_h)^{\frac{\alpha}{2}}$ and $A_h$, where $A_h$ is as \eqref{long-range}, are divergent as $h\to 0$ in the space of bounded linear operators on $L^2(\mathbb{R}^d)$ in the uniform operator topology. On the other hand, $(-\Delta_h)^{\frac{\alpha}{2}},\ A_h$ converge to $(-\Delta)^{\frac{\alpha}{2}}$ strongly in $L^2$ in the Schwartz class.
\end{proposition}
\begin{proof}
By direct computation,
\begin{equation*}
\| \sigma_h \|_{L^\infty(\frac{\mathbb{T}^d}{h})},\ \|m_h \|_{L^\infty(\frac{\mathbb{T}^d}{h})} \gtrsim_d h^{-\alpha}, 
\end{equation*}
and hence the divergence by \Cref{multiplier}. The statements on strong convergence in $L^2$ follows as \cite[Lemma 3.9]{kirpatrick}, noting that
\begin{equation*}
    \sigma_h(\xi),\ m_h(\xi) \xrightarrow[h\rightarrow 0]{}|\xi|^\alpha,\ \forall \xi \in \mathbb{R}^d,
\end{equation*}
followed by the Dominated Convergence Theorem to interchange the limit as $h$ tends to zero and the integral in $\xi$, justified by
\begin{equation*}
 |\sigma_h(\xi)|,\ |m_h(\xi)| \lesssim |\xi|^\alpha,   
\end{equation*}
independent of $h>0$.
\end{proof}
A potential issue with $m_h$ in \eqref{multiplier2} is that the Euclidean metric $|\cdot|$ does not yield the physically-relevant distance between two points on $h\mathbb{Z}^d$ when $d \geq 2$. Define $m_h^q(\xi) = h^d \sum\limits_{z \in h\mathbb{Z}^d\setminus \{0\}} \frac{1-\cos \xi\cdot z}{|z|_q^{d+\alpha}}$ where $|\cdot|_q$ denotes the $l^q$ norm for $q \in [1,\infty]$.
\begin{proposition}
For all $\xi\in\mathbb{R}^d$, $m_h^q(\xi)\xrightarrow[h\rightarrow 0]{} c |\xi|^\alpha$ for some non-zero constant $c$ if and only if $q=2$. For all $q \in [1,\infty]\setminus\{2\}$, an operator $A_h^q$ defined by $m_h^q$ is non-negative and self-adjoint on $L^2(\mathbb{R}^d)$, defined on the dense domain $H^\alpha(\mathbb{R}^d)$.
\end{proposition}
\begin{proof}
Let $\xi = |\xi|\xi^\prime$ where $\xi^\prime \in S^{d-1}$, the unit sphere. Let $\rho(\xi^\prime) \in \mathbb{R}^{d \times d}$ be a rotation operator that takes the $d^{\text{th}}$ standard basis vector to $\xi^\prime$, i.e., $\xi^\prime = \rho(\xi^\prime)e_d$. Then $\alpha<2$ justifies the limit of Riemann sum, and by changing variables,
\begin{equation*}
    \lim_{h\rightarrow 0}m_h^q(\xi) = \int_{\mathbb{R}^d} \frac{1-\cos \xi \cdot z}{|z|_q^{d+\alpha}}dz = |\xi|^\alpha \int \frac{1-\cos z_d}{|\rho(\xi^\prime)z|_q^{d+\alpha}}dz.
\end{equation*}
A priori, the integral in the last expression, call it $I(\xi)$, reduces to a continuous function on $S^{d-1}$ that is constant if and only if $q=2$, observing the norm-invariance under rotation if and only if $q=2$. The domain of $A_q^h$ consists of $f \in L^2(\mathbb{R}^d)$ such that $|\xi|^\alpha I(\xi)\widehat{f} \in L^2$. Observing that $\inf\limits_{\xi\in\mathbb{R}^d}|I(\xi)|>0$ due to the norm equivalence of $\{|\cdot|_q\}$ and
\begin{equation*}
    I(\xi) \gtrsim \int \frac{1-\cos z_d}{|\rho(\xi^\prime)z|^{d+\alpha}}dz = \int \frac{1-\cos z_d}{|z|^{d+\alpha}}dz = c_{d,\alpha}^{-1}>0,
\end{equation*}
it follows that $D(A_h^q) = H^\alpha(\mathbb{R}^d)$. That $A_h^q$ is non-negative and self-adjoint follows from $I \geq 0$.

\end{proof}
The continuum limit in higher dimensions, therefore, depends on the geometry of the underlying discrete model. This would potentially lead to complications on a spatial domain with an irregular lattice structure, which we leave as an open-ended thought. To this end, our analysis is restricted to $(-\Delta_h)^{\frac{\alpha}{2}}$.

To show the main result, linear dispersive estimates of the discrete evolution are developed. Let $\psi \in C^\infty_c((-2\pi,2\pi);[0,1])$ be an even function where $\psi = 1$ for $\xi \in [-\pi,\pi]$ and let $\eta(\xi):= \psi(|\xi|)-\psi(2|\xi|)$. For dyadic integers $N \leq 1$, define Littlewood-Paley projections given by
\begin{equation*}
    P_N = P_{N,h} := \mathcal{F}^{-1} \eta\Big(\frac{h\xi}{N}\Big)\mathcal{F},
\end{equation*}
where $\mathcal{F}$ is the Fourier transform on $\mathbb{R}^d$. Since $\xi \in \frac{\mathbb{T}^d}{h}$, $P_N$ is a smooth projector onto $\frac{\pi}{2}\frac{N}{h}\leq|\xi|\leq 2\pi \frac{N}{h}$ and altogether resolves the identity
\begin{equation*}
    \sum_{N \leq 1} P_N = Id.
\end{equation*}
The sum has an upper bound in $N$ since $h\xi = O_d(1)$. 

Adopting the notations in \cite{borovyk2017klein}, define subsets of $M:=\mathbb{T}^2\setminus\{0\}$ given by
\begin{equation*}
    K_3 = \{(\pm \frac{\pi}{2},\pm \frac{\pi}{2})\},\: K_2 = \{\xi \in M\setminus K_3: \det D^2w(\xi)=0\},\: K_1 = M \setminus (K_2 \cup K_3),
\end{equation*}
where $w(\xi) = \Big(\sum\limits_{i=1}^d \sin^2(\frac{\xi_i}{2})\Big)^{\frac{\alpha}{2}}$. The main proposition concerns a family of frequency-localized dispersive estimates with sharp time decay. Furthermore the lower bounds of implicit constants blow up both in the wave and Schr\"odinger limit. 

\begin{proposition}\label{conjecture} For all $t>0,\: N \leq 1,\: 0<h\leq 1,\: 1<\alpha<2$, there exists $0<C_i(\alpha)<\infty,\: i=1,2,3$ such that
\begin{equation}\label{prove}
\|U_h(t)P_N f \|_{L^\infty_h} \leq
\begin{cases}
C_3(\alpha)(\frac{N}{h})^{2-\frac{3}{4}\alpha}|t|^{-\frac{3}{4}}\| f \|_{L^1_h}, & supp(\eta(\frac{\cdot}{N}))\cap K_3 \neq \emptyset\\
C_2(\alpha)(\frac{N}{h})^{2-\frac{5}{6}\alpha}|t|^{-\frac{5}{6}}\| f \|_{L^1_h}, & supp(\eta(\frac{\cdot}{N}))\cap (K_2\setminus K_3) \neq \emptyset\\
C_1(\alpha)(\frac{N}{h})^{2-\alpha}|t|^{-1}\| f \|_{L^1_h}, & supp(\eta(\frac{\cdot}{N}))\cap (K_1\setminus K_2 \cup K_3) \neq \emptyset,
\end{cases}
\end{equation}
and
\begin{equation}\label{lower bound}
C_3(\alpha) \gtrsim_\eta (2-\alpha)^{-\frac{1}{4}},\:C_2(\alpha) \gtrsim_\eta (\alpha-1)^{\frac{2}{3}-\frac{5\alpha}{12}},\:C_1(\alpha) \gtrsim_\eta (\alpha-1)^{-\frac{1}{2}}.   
\end{equation}
\end{proposition}
For more details on the domain of $N\in 2^\mathbb{Z}$ that satisfy \eqref{prove}, see \eqref{N}. By interpolating the estimates in \eqref{prove}, one obtains
\begin{corollary}\label{conjecture2}
Assume the hypotheses of \Cref{conjecture}. Then,
\begin{equation*}
\|U_h(t)P_N f \|_{L^\infty_h} \lesssim_\alpha (\frac{N}{h})^{2-\frac{3}{4}\alpha}|t|^{-\frac{3}{4}}\| f \|_{L^1_h}.   
\end{equation*}
\end{corollary}
\begin{remark}
Assuming \Cref{conjecture}, it is straightforward to obtain the Strichartz estimates for the linear evolution by averaging in $t,N$, which we briefly describe. Suppose $\|U_h(t)P_N  \|_{L^1_h \rightarrow L^\infty_h} \lesssim_\alpha (\frac{N}{h})^{\beta}|t|^{-\sigma}$ for some $\beta,\sigma>0$. Define the Strichartz pair $(q,r) \in [2,\infty]^2$ by the relation
\begin{equation}\label{strichartz pair}
\frac{1}{q}+\frac{\sigma}{r} = \frac{\sigma}{2},\: (q,r,\sigma) \neq (2,\infty,1).    
\end{equation}
As in \cite[p.1127]{cho2011remarks}, define $\Tilde{U}(t) = P_N U_h\Big( (\frac{N}{h})^{\frac{\beta}{\sigma}}t\Big)P_{\sim N}$ where $P_{\sim N} := P_{N/2}+P_N+P_{2N}$. Then $\{\Tilde{U}(t)\}_{t\in\mathbb{R}}$ satisfies the hypotheses of \cite[Theorem 1.2]{keel1998endpoint} from which follows
\begin{equation*}
    \| U_h(t) P_N f \|_{L^q_t L^r_h} \lesssim_{q,r} (\frac{N}{h})^{\beta (\frac{1}{2}-\frac{1}{r})}\| P_N f \|_{L^2_h} \simeq \| P_N |\nabla_h|^{\beta (\frac{1}{2}-\frac{1}{r})} f \|_{L^2_h}.
\end{equation*}
Squaring both sides and summing in $N$,
\begin{equation}\label{strichartz}
\| U_h(t) f \|_{L^q_t L^r_h} \lesssim \Big(\sum_{N \leq 1} \| U_h(t) P_N f \|_{L^q_t L^r_h}^2 \Big)^{\frac{1}{2}} \lesssim \Big(\sum_{N \leq 1} \| P_N |\nabla_h|^{\beta (\frac{1}{2}-\frac{1}{r})} f \|_{L^2_h}^2 \Big)^{\frac{1}{2}} \simeq \|  |\nabla_h|^{\beta (\frac{1}{2}-\frac{1}{r})} f \|_{L^2_h},     
\end{equation}
for $r \in [2,\infty)$ where the first inequality follows from the Littlewood-Paley inequality on the lattice \cite[Theorem 4.2]{hong2018uniform}. As an example, \Cref{conjecture2} asserts $(\beta,\sigma) = (2-\frac{3}{4}\alpha,\frac{3}{4})$ and hence by \eqref{strichartz},
\begin{equation*}
   \| U_h(t) f \|_{L^q_t L^r_h} \lesssim_{q,r,\alpha} \|  |\nabla_h|^{(2-\frac{3}{4}\alpha) (\frac{1}{2}-\frac{1}{r})} f \|_{L^2_h}.  
\end{equation*}
The derivative loss occurs for $\alpha<\frac{8}{3}$ or $\alpha<2$ on $h\mathbb{Z}^2$ (for all $h>0$) or $\mathbb{R}^2$, respectively.
\end{remark}

For $v \in \mathbb{R}^d$, define $\Phi_v(\xi) = v\cdot \xi - w(\xi)$ for $\xi \in \mathbb{R}^d$ and let $\zeta \in C^\infty_c(\mathbb{R}^d)$ be a test function. Consider
\begin{equation*}
    J=J_{\Phi_v,\zeta}(\tau) := \int_{\mathbb{R}^d} e^{i\tau \Phi_v(\xi)}\zeta(\xi)d\xi,
\end{equation*}
where $\tau>0$ without loss of generality, for $\tau<0$ amounts to taking the complex conjugate of $J_{\Phi_v,\zeta}$. To show \eqref{prove}, observe that
\begin{equation*}
\|U_h(t)P_N f \|_{L^\infty_h} =\|K_{t,N,h} \ast f \|_{L^\infty_h} \leq \|K_{t,N,h} \|_{L^\infty_h} \| f \|_{L^1_h}    
\end{equation*}
by the Young's inequality applied to the convolution in $h\mathbb{Z}^d$ where
\begin{equation*}
K_{t,N,h}(x) = (2\pi)^{-d} \int_{\frac{\mathbb{T}^d}{h}} e^{i\Big\{x \cdot \xi - t\Big(\frac{4}{h^2}\sum\limits_{i=1}^d \sin^2(\frac{h\xi_i}{2})\Big)^{\frac{\alpha}{2}}\Big\}}\eta\Big(\frac{h\xi}{N}\Big)d\xi.   
\end{equation*}
Change variables $\xi \mapsto \frac{\xi}{h}$ and define $\tau = \frac{2^\alpha t}{h^\alpha},\: v=\frac{x}{h\tau}$ to obtain
\begin{equation*}
    K_{t,N,h}(x)= (2\pi h)^{-d} \int_{\mathbb{T}^d} e^{i\tau(v\cdot \xi - w(\xi))} \eta\Big(\frac{\xi}{N}\Big)d\xi.
\end{equation*}
A priori since $x \in h\mathbb{Z}^d$, it follows that $v \in \tau^{-1}\mathbb{Z}^d$, which we consider as a subset of $\mathbb{R}^d$. If $\sigma_0>0$ is the sharp decay rate for $J_{\Phi_v,\zeta}$ in the sense that
\begin{equation}\label{decay}
    \sup_{v\in \mathbb{R}^d}|J_{\Phi_v,\zeta}(\tau)|  \leq C(\zeta) |\tau|^{-\sigma_0}
\end{equation}
holds for all $\tau \in \mathbb{R}\setminus \{0\}$ for some $C(\zeta)>0$ and no bigger $\sigma^\prime>0$ satisfies \eqref{decay}, then
\begin{equation*}
    \| K_{t,N,h} \|_{L^\infty_h} \leq (2\pi h)^{-d} C\Big(\eta(\frac{\cdot}{N})\Big)|\tau|^{-\sigma_0} = (2\pi)^{-d} 2^{-\sigma_0 \alpha}\frac{C\Big(\eta(\frac{\cdot}{N})\Big)}{h^{d-\sigma_0 \alpha}}|t|^{-\sigma_0}. 
\end{equation*}
Hence our goal reduces to obtaining \eqref{decay} for a dyadic family of Littlewood-Paley functions. Outside of a neighborhood of the origin, $\Phi_v$ is analytic, and therefore the major contributions to $J$ are due to critical points $\xi \in \mathbb{R}^d$ that satisfy $\nabla \Phi_v(\xi)=0$, or equivalently, $v = \nabla w (\xi)=: v_\xi$. In any arbitrary dimension,
\begin{equation}\label{group velocity}
\begin{split}
\nabla w(\xi) &= \frac{\alpha}{4} w(\xi)^{-(\frac{2}{\alpha}-1)}(\sin \xi_1,\dots,\sin \xi_d),\\
|\nabla w(\xi)|^2 &= \frac{\alpha^2}{16} \frac{\sum\limits_{i=1}^d \sin^2 \xi_i}{(\sum\limits_{i=1}^d \sin^2 \frac{\xi_i}{2})^{2-\alpha}}.
\end{split}
\end{equation}
For any $\alpha >1$, $|\nabla w(\xi)|\xrightarrow[\xi\rightarrow 0]{}0$ and $\sup\limits_{\xi \in \mathbb{T}^d}|\nabla w(\xi)| = C(\alpha,d) \in (0,\infty)$. With a slight abuse of notation, define $\nabla w:\mathbb{T}^d\rightarrow \mathbb{R}^d$ where $\nabla w (0)=0
$. Then $\nabla w$ is continuous and its compact image defines a light cone. If $|v|\gg_{\alpha,d,h} 1$ (spacelike event), then $J$ decays faster than $\tau^{-n}$ for any $n \in \mathbb{N}$, by integration by parts, with the implicit constant dependent on $n$ and the distance between $v$ and the light cone (see \Cref{non-stationary}). Inside the light cone (including the boundary), $J$ undergoes an algebraic decay due to critical points. For such $v$, it is generically true that the corresponding critical point(s) $\xi \in \mathbb{T}^d$ are non-degenerate, and therefore $J$ decays as $\tau^{-\frac{d}{2}}$. However there exists a low-dimensional subset of $\mathbb{T}^d$ that retards even further the decay rate of $\frac{d}{2}$. We consider this problem of resolution of singularities for $d=2$.

To systematically study the decay and asymptotics of $J$ as a function of $v$ and $\zeta$, consider the Taylor series expansion of $\Phi_v$. Let $\xi \in \mathbb{T}^2\setminus\{0\}$ and $v_\xi = \nabla w(\xi)$. Consider $\Phi_{v_\xi}$ so that $\nabla \Phi_{v_\xi} (\xi)=0$. Pick $\zeta \in C^\infty_c$ around $\xi$ such that $\xi$ is the unique critical point in the support. Then $J_{\Phi_{v_\xi},\zeta}$ has an asymptotic expansion
\begin{equation}\label{asymptotics}
    J_{\Phi_{v_\xi},\zeta} =  d_0(\zeta)\tau^{-\sigma_0}+o(\tau^{-\sigma_0}),
\end{equation}
as $\tau \rightarrow \infty$ where $\sigma_0$, or the \textit{oscillatory index}, is chosen to be the minimal number such that for any neighborhood of $\xi$, say $U$, there exists $\zeta_U \in C^\infty_c(U)$ such that $d_0(\zeta_U) \neq 0$; in particular, $\sigma_0$ depends only on the phase, not the smooth bump function. Under some hypotheses, $\sigma_0,d_0$ are deduced from the higher order Taylor expansion of $\Phi_{v_\xi}$ (see \Cref{greenblatt}), a process that we briefly describe.

Let $\Phi$ be a real-valued analytic function on a small neighborhood of the origin. Assume $\Phi(0)=0,\:\nabla \Phi(0)=0$ and therefore the Taylor expansion of $\Phi$ at the origin in the multi-index notation is
\begin{equation*}
    \Phi(x) = \sum_{|\alpha|\geq 2} c_\alpha x^\alpha = \sum_{|\alpha|\geq 2} \frac{\partial_\alpha \Phi(0)}{\alpha !} x^\alpha.
\end{equation*}
Define the Taylor support $\mathcal{T} = \{\alpha \in \mathbb{N}^d: c_\alpha \neq 0\}$ and assume that $\Phi$ is of finite type, i.e., $\mathcal{T} \neq \emptyset$. Define the \textit{Newton polyhedron} of $\Phi$, call it $\mathcal{N}$, to be the convex hull of 
\begin{equation*}
    \bigcup_{\alpha \in \mathcal{T}} \alpha + \mathbb{R}_+^d = \bigcup_{\alpha \in \mathcal{T}} \alpha + \{x \in \mathbb{R}^d: x_i \geq 0\}, 
\end{equation*}
and the \textit{Newton diagram} $\mathcal{N}_d$ to be the union of all compact faces of $\mathcal{N}$. Let $\mathcal{N}_{pr}$, the principal part of Newton diagram, be the subset of $\mathcal{N}_d$ that intersects the bisectrix $\{x_1 = x_2=\cdots = x_d\}$. Define the principal part of $\Phi$ (or the normal form) as
\begin{equation*}
\Phi_{pr}(x) = \sum_{|\alpha|\geq 2,\: \alpha \in \mathcal{N}_{pr}}c_\alpha x^\alpha.    
\end{equation*}
Let $d=d(\Phi) = \inf \{t: (t,t,\dots,t)\in \mathcal{N}\}$ be the distance from the origin to $\mathcal{N}$. Since $\Phi$ is of finite type, $0<d<\infty$. Note that $\mathcal{T}$ is not invariant under analytic coordinate transformations. Let $d_x$ be the distance computed in the $x$ coordinate system and define the \textit{height} of $\Phi$ as $h(\Phi)=\sup\limits_{x} d_x$ where the supremum is over all analytic coordinate systems. The coordinate system $(x)$ is \textit{adapted} if $d_x = h$. In $\mathbb{R}^2$, see \cite[Proposition 0.7,0.8]{varchenko1976newton} for sufficient conditions for $(x)$ to be adapted. An adapted system need not be unique. To obtain the asymptotics of oscillatory integrals, we work in a \textit{superadapted} coordinate system defined specifically in dimension two in \cite{greenblatt2009asymptotic} as a coordinates system in which $\Phi_{pr}(x,\pm 1)$ have no real roots of order greater than or equal to $d_x(\Phi)$, possibly except $x=0$. In particular, if $d(\Phi)>1$ and $\Phi_{pr}(x,\pm 1)$ is a quadratic polynomial with no repeated roots, then $(x,y)$ is superadapted.

See the introductions of \cite{borovyk2017klein,greenblatt2009asymptotic,ikromov2011uniform,varchenko1976newton}, from which this paper adopts all relevant terminologies, for a brief survey of the relationship between oscillatory integrals and Newton polyhedra. To illustrate these ideas, consider an example. Let $\Phi(x,y) = x^2+y^2+x^3$. Then 
\begin{equation*}
\begin{split}
\mathcal{T} &= \{(2,0),(0,2),(3,0)\}\\
\mathcal{N}&=\{(x,y)\in \mathbb{R}^2: x+y \geq 2,\: x,y \geq 0\},\: \mathcal{N}_d = \mathcal{N}_{pr} = \{(x,y)\in \mathbb{R}^2: x+y =2,\: 0 \leq x \leq 2\}\\
\Phi_{pr}(x,y) &= x^2+y^2,\: d_{(x,y)} =1.
\end{split}    
\end{equation*}
Since $\Phi_{pr}(x,\pm 1) = x^2+1$ has no real root, the given coordinates system is superadapted.
\section{Continuum limit.}\label{continuum limit}
The proof of \Cref{mainthm} is given.
\begin{lemma}\label{series}
Let $\beta \in [0,1],\: p>1,\: d \in \mathbb{N}$. The implicit constants in the following estimates are independent of $h>0$ and dependent only on $\beta,d$.
\begin{enumerate}
    \item $\| d_h f \|_{H^\beta_h} \lesssim \| f \|_{H^\beta(\mathbb{R}^d)}$.
    \item $\| p_h f_h \|_{H^\beta(\mathbb{R}^d)} \lesssim \| f_h \|_{H^\beta_h}$.
    \item $\| p_h d_h f - f \|_{L^2(\mathbb{R}^d)} \lesssim h^\beta \| f \|_{H^\beta(\mathbb{R}^d)}$.
    \item $\| p_h U_h(t) f_h - U(t)u_0 \|_{L^2(\mathbb{R}^d)} \lesssim h^{\frac{\beta}{1+\beta}}|t|(\| f_h \|_{H^\beta_h} + \| u_0 \|_{H^\beta(\mathbb{R}^d)})+ \|p_h f_h-u_0 \|_{L^2(\mathbb{R}^d)}$.\\
    If $f_h$ is the discretization of $u_0$, i.e., $f_h= u_{0,h}$, then $\| p_h U_h(t) u_{0,h} - U(t)u_0\|_{L^2(\mathbb{R}^d)} \lesssim \langle t \rangle h^{\frac{\beta}{1+\beta}}\| u_0 \|_{H^\beta(\mathbb{R}^d)}$.
    \item $\| p_h (|u_h|^{p-1}u_h) - |p_h u_h|^{p-1}p_h u_h \|_{L^2(\mathbb{R}^d)} \lesssim h^\beta \| u_h \|_{L^\infty_h}^{p-1} \| u_h \|_{H^\beta_h}$.
\end{enumerate}
\end{lemma}
\begin{proof}
See \cite{kirpatrick,hong2019strong}.
\end{proof}

\begin{lemma}\label{discrete linfty}
Let $p \geq 3$ and $\max(\frac{8}{7}, \frac{2(p-1)}{p+1})<\alpha<2$. There exists $\delta = \delta(\alpha,p)>0$ sufficiently small such that a Strichartz pair $(q,r)$, defined by \eqref{strichartz pair} with $\sigma = \frac{3}{4}$ and $q:= \frac{2\alpha}{2-\alpha+\delta}$,  satisfies $p-1<q$ and yields the uniform $L^\infty_h$ estimate given by
\begin{equation*}
    \| U_h(t)f_h \|_{L^q_t L^\infty_h} \lesssim \| f_h \|_{H^{\frac{\alpha}{2}}_h}.
\end{equation*}
\end{lemma}
\begin{proof}
Let $s = \frac{\alpha}{2}-(2-\frac{3}{4}\alpha)(\frac{1}{2}-\frac{1}{r})$. Then it can be verified by direct computation that $s>\frac{2}{r}$ using $\delta>0$. Hence by the Sobolev embedding and \eqref{strichartz}, respectively,
\begin{equation*}
    \| U_h(t) f_h \|_{L^q_t L^\infty_h} \lesssim \| U_h(t) f_h \|_{L^q_t W^{s,r}_h} \lesssim \| f_h \|_{H^{\frac{\alpha}{2}}_h}.
\end{equation*}
Furthermore it can be directly verified that $(q,r) \in [2,\infty]^2$. From the Strichartz pair relation and the definition of $q$, we have $r \leq \infty$ iff $\alpha > \frac{8}{7}$. Lastly by choosing $\delta>0$ sufficiently small, $p-1<q$ is satisfied.
\end{proof}
\begin{remark}
From \eqref{balance}, we have
\begin{equation*}
    2 \leq p-1 < \frac{2\alpha}{d-\alpha},
\end{equation*}
from which the definition of $q$ in \Cref{discrete linfty} is motivated.
\end{remark}

\begin{lemma}
Assume the hypothesis of \Cref{mainthm} and let $q$ be given by \Cref{discrete linfty}. Given $u_0 \in H^{\frac{\alpha}{2}}(\mathbb{R}^2)$, let $T \simeq \| u_0 \|_{H^{\frac{\alpha}{2}}}^{-\frac{1}{\frac{1}{p-1}-\frac{1}{q}}}$. By \Cref{dnls}, let $u,u_h$ be the well-posed solutions corresponding to initial data $u_0,u_{0,h}$. Then $u,u_h$ satisfy
\begin{equation}\label{uniformestimates}
\begin{split}
\| u \|_{L^\infty_{t\in [0,T]}H^{\frac{\alpha}{2}}_x} + \| u \|_{L^q_{t\in [0,T]}L^\infty_x} &\lesssim \| u_0 \|_{H^{\frac{\alpha}{2}}},\\
\| u_h \|_{L^\infty_{t\in [0,T]}H^{\frac{\alpha}{2}}_h} + \| u_h \|_{L^q_{t\in [0,T]}L^\infty_h} &\lesssim \| u_{0,h} \|_{H^{\frac{\alpha}{2}}_h}.
\end{split}
\end{equation}
\end{lemma}
\begin{proof}
The estimate for $u$ is derived from the proof of local well-posedness by the contraction mapping argument. Similarly the time of existence for the discrete evolution that ensures the estimate \eqref{uniformestimates} is $T_h \sim \| u_{0,h} \|_{H^{\frac{\alpha}{2}}_h}^{-\frac{1}{\frac{1}{p-1}-\frac{1}{q}}}$. By \Cref{series}, $T_h \gtrsim T$ uniformly in $h$, and therefore $u,u_h$ are well-defined on $[0,T]$.
\end{proof}

\begin{proof}[Proof of \Cref{mainthm}]
Let $u_0 \in H^{\frac{\alpha}{2}}(\mathbb{R}^2)$. There exists $T \sim \| u_0 \|_{H^{\frac{\alpha}{2}}}^{-\frac{1}{\frac{1}{p-1}-\frac{1}{q}}}$ and $u \in C([0,T];H^{\frac{\alpha}{2}}(\mathbb{R}^2))$, unique in a (smaller) Strichartz space (see \cite[Theorem 1.1]{1534-0392_2015_6_2265}), satisfying
\begin{equation*}
  u(t) = U(t)u_0 - i\mu \int_0^t U(t-s)(|u|^{p-1}u)(s)ds.  
\end{equation*}
For $h>0$, consider $u_{0,h} = d_h u_0 \in L^2_h$ and the global solution $u_h \in C^1([0,\infty);L^2_h)$ given by \Cref{dnls}. Similarly as above,
\begin{equation*}
p_h u_h(t) = p_h U_h(t) u_{0,h} -i\mu \int_0^t p_h U_h(t-s)(|u_h|^{p-1}u_h)(s)ds. 
\end{equation*}
The difference $p_h u_h(t) - u(t)$ is given by
\begin{equation*}
\begin{split}
&=p_h U_h(t)u_{0,h} - U(t)u_0\\
&\hspace{0.3cm}- i\mu \int_0^t \Big(p_h U_h(t-s) - U(t-s)p_h\Big)(|u_h|^{p-1}u_h)(s)ds\\
&\hspace{0.3cm}-i\mu \int_0^t U(t-s)\Big(p_h(|u_h|^{p-1}u_h)(s) - |p_hu_h|^{p-1} p_h u_h(s)\Big)ds\\
&\hspace{0.3cm} -i\mu \int_0^t U(t-s)\Big(|p_hu_h|^{p-1}p_hu_h(s)-|u|^{p-1}u(s)\Big)ds=: I+II+III+IV.
\end{split}
\end{equation*}
Following the proof of \cite[Theorem 1.1]{hong2019strong}, we have
\begin{equation*}
\begin{split}
\| I \|_{L^2} &\lesssim h^{\frac{\alpha}{2+\alpha}} \langle t \rangle \| u_0 \|_{H^{\frac{\alpha}{2}}}\\
\| II \|_{L^2},\: \| III \|_{L^2} &\lesssim h^{\frac{2}{2+\alpha}} \langle t \rangle^2 \| u_0 \|_{H^{\frac{\alpha}{2}}}^p\\
\| IV \|_{L^2} &\lesssim \int_0^t (\| u_h(s)\|_{L^\infty_h}+\| u(s)\|_{L^\infty_x})^{p-1} \| p_h u_h (s) - u(s)\|_{L^2}ds.
\end{split}
\end{equation*}
and altogether,
\begin{equation*}
\begin{split}
\| p_h u_h (t) - u(t)\|_{L^2} \lesssim h^{\frac{\alpha}{2+\alpha}}\langle t \rangle^2 ( \| u_0 \|_{H^{\frac{\alpha}{2}}}+\| u_0 \|_{H^{\frac{\alpha}{2}}}^p) +\int_0^t (\| u_h(s)\|_{L^\infty_h}+\| u(s)\|_{L^\infty_x})^{p-1} \| p_h u_h (s) - u(s)\|_{L^2}ds.
\end{split}
\end{equation*}
By the Gronwall's inequality,
\begin{equation}\label{gronwall}
\begin{split}
\| p_h u_h (t) - u(t)\|_{L^2} \lesssim h^{\frac{\alpha}{2+\alpha}}\langle t \rangle^2 ( \| u_0 \|_{H^{\frac{\alpha}{2}}}+\| u_0 \|_{H^{\frac{\alpha}{2}}}^p) e^{\int_0^t (\| u_h(s)\|_{L^\infty_h}+\| u(s)\|_{L^\infty_x})^{p-1} ds}.
\end{split}
\end{equation}
Applying \eqref{uniformestimates} to \eqref{gronwall}, we obtain \eqref{error estimate} where $C_1,C_2>0$ depend on various parameters including $\| u_0 \|_{H^{\frac{\alpha}{2}}}$, but not $h$. This completes the proof.
\end{proof}
\section{Proof of \Cref{conjecture}.}\label{main proof}
Let $\alpha \in (1,2)$ unless otherwise specified. For $\zeta \in C^\infty_c(M)$, the quantity $J_{\Phi_v,\zeta}$ is at worst a non-degenerate integral almost everywhere with respect to the Lebesgue measure on $\mathbb{R}^2_v$. 

\begin{lemma}\label{non-stationary}
Let $\zeta \in C^\infty_c(U)$ where $U \subseteq \mathbb{R}^d \setminus \{0\}$. For $v \in \mathbb{R}^d$, suppose $\inf\limits_{\xi\in U}|v-\nabla w(\xi)| \geq m>0$ on $U$. Then $|J_{\Phi_v,\zeta}| \lesssim_{n,m,\zeta} |\tau|^{-n}$ for all $\tau \in \mathbb{R}\setminus\{0\},\: n \in \mathbb{N}$.
\end{lemma}

Non-degenerate critical points are treated by the method of stationary phase. A well-known asymptotics \cite[Chapter 8, Proposition 6]{stein1993harmonic} is given.

\begin{lemma}\label{stationary}
Let $\xi \in M,\: v \in \mathbb{R}^2$ satisfy $v = \nabla w(\xi)$ and $\det D^2 w(\xi) \neq 0$. Then there exists a small neighborhood around $\xi$ such that for all $\zeta \in C^\infty_c$ supported in the neighborhood,
\begin{equation*}
    J_{\Phi_v,\zeta} = a_0\tau^{-1} + o(\tau^{-1}),
\end{equation*}
as $\tau \rightarrow \infty$ where
\begin{equation*}
a_0 = e^{i\frac{\pi}{4}sgn D^2w(\xi)}e^{i\tau w(\xi)} \zeta(\xi)\sqrt{\frac{2\pi}{|\det D^2w(\xi)|}}.    
\end{equation*}
\end{lemma}

If $D^2w(\xi)$ is singular, i.e., if $\xi \in M$ satisfies $H(\xi,\alpha):=\det D^2w(\xi)=0$, then the asymptotic formula of \Cref{stationary} is not applicable; $\xi=0$ is not considered since the oscillatory integrals from \Cref{conjecture} have test functions supported outside of the origin. Note that
\begin{equation}\label{hessian}
    D^2 w(\xi) =-\frac{\alpha}{16w(\xi)^{\frac{4}{\alpha}-1}}\begin{pmatrix}
\alpha\cos^2 \xi_1 +2(\cos \xi_2-2)\cos \xi_1 + 2- \alpha & (2-\alpha)\sin \xi_1 \sin \xi_2\\
(2-\alpha)\sin \xi_1 \sin \xi_2 & \alpha\cos^2 \xi_2 +2(\cos \xi_1-2)\cos \xi_2 + 2- \alpha
\end{pmatrix}
\end{equation}
is well-defined for $\xi \in M$ and blows up (in the sense of determinant) as $\xi \rightarrow 0$. It can be shown directly that $D^2 w(\xi)$ is not the zero matrix for any $\xi \in M$. If $D^2w(\xi)$ is of full rank, then the decay of $J_{\Phi_{v_\xi},\zeta}$ can be analyzed via \Cref{non-stationary,stationary}, and henceforth suppose $rank(D^2w(\xi))=1$. Then $H(\xi,\alpha) = \Tilde{h}(\xi,\alpha) h(\xi,\alpha)$ where
\begin{equation*}
\begin{split}
\Tilde{h}(\xi,\alpha)&=\frac{\alpha^2}{128w(\xi)^{\frac{8}{\alpha}-2}}(\cos\xi_1+\cos\xi_2-2)\\
h(\xi,\alpha)&=\alpha\cos \xi_1 \cos \xi_2 (\cos \xi_1+\cos \xi_2) -4\cos \xi_1 \cos \xi_2 + (2-\alpha)(\cos \xi_1+\cos \xi_2).
\end{split}
\end{equation*}
Since $h,\Tilde{h}$ are symmetric under $\xi_1\mapsto \pm \xi_1,\:\xi_2 \mapsto \pm \xi_2,\: (\xi_1,\xi_2)\mapsto (\xi_2,\xi_1)$, the domain of analysis could be restricted to the first quadrant either above or below the identity $\xi_1=\xi_2$. By definition, $\Tilde{h}$ is nonzero on $M$ and therefore the roots of $H$ correspond to those of $h$. Following the approach in \cite{borovyk2017klein}, the representation of $h$ is in polynomials under the change of variables $a=\cos \xi_1,\: b= \cos \xi_2$, and is given by
\begin{equation*}
    h(a,b,\alpha)=\alpha ab (a+b) -4 ab +(2-\alpha)(a+b).
\end{equation*}

For brevity, let
\begin{equation*}
\begin{split}
E_\alpha&= \{(a,b)\in [-1,1)^2: h(a,b,\alpha)=0\},\\
E &= \bigcup\limits_{\alpha \in (1,2)}E_\alpha.
\end{split}
\end{equation*}
For $\alpha \in [1,2)$, since $\nabla_{(a,b)}h \neq 0$ on $[-1,1)^2$, $E_\alpha$ is a smooth one-dimensional embedded submanifold by the Implicit Function Theorem. This is false for $\alpha=2$; a cusp $(a,b)=(0,0)$, which corresponds to $\xi=(\pm \frac{\pi}{2},\pm \frac{\pi}{2})$, appears when $\alpha=2$ since $\nabla_{(a,b)}h(0,0) = 0$. By direct computation, $E_{\alpha_1}\cap E_{\alpha_2} = \{(0,0)\}$ for all $\alpha_1,\alpha_2 \in [1,2]$. Hence there exists a smooth map $E \setminus \{0\} \ni (a,b)\mapsto \alpha \in (1,2)$ satisfying $h(a,b,\alpha)=0$ by the Implicit Function Theorem, observing that $\partial_\alpha h(a,b,\alpha) = (a+b)(ab-1)$ is non-vanishing on $E \setminus \{0\}$. Observe that $E_\alpha$ consists of two connected components; one component, say $\Gamma^1_{(a,b)}(\alpha)$, passes through the origin  whereas the other, say $\Gamma^2_{(a,b)}(\alpha)$, does not. As $\alpha \rightarrow 1+$, $\Gamma^2_{(a,b)}(\alpha)$ becomes arbitrarily close to $(a,b)=(1,1)$ whose corresponding point in $\mathbb{T}^2$, the origin, is not in $M$. In fact,  
\begin{equation*}
\bigcap_{\alpha_0 \in (1,2]}\bigcup_{\alpha \in [1,\alpha_0)} E_\alpha=\Gamma^1_{(a,b)}(1),\:  \bigcap_{\alpha_0\in [1,2)}\bigcup_{\alpha\in (\alpha_0,2]}E_\alpha=\{ab=0\}.
\end{equation*}
See \Cref{contour} for the contour plots of $E_\alpha$. 

\begin{figure}[h]
\centering
\includegraphics[scale=0.43]{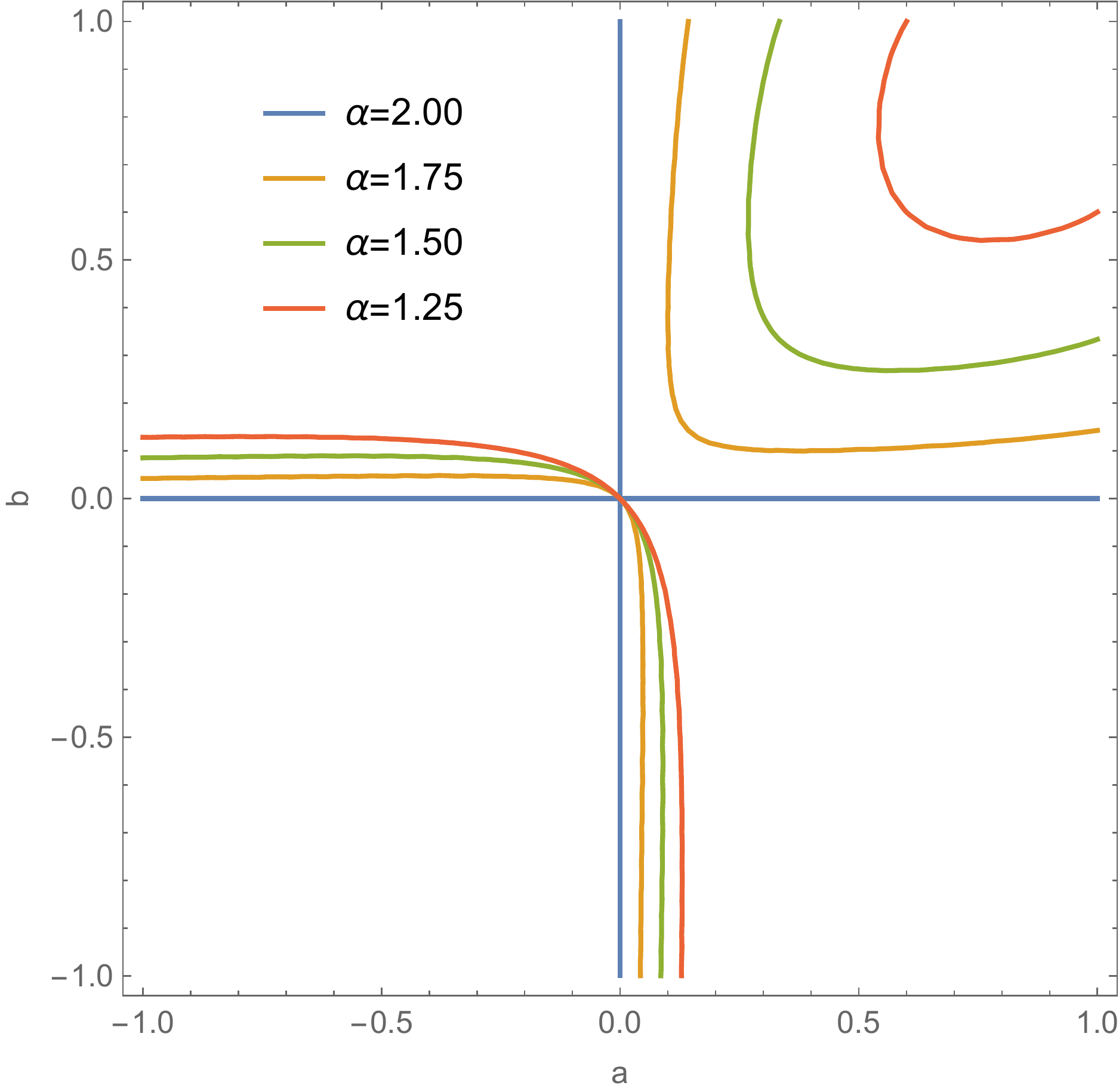}
\includegraphics[scale=0.42]{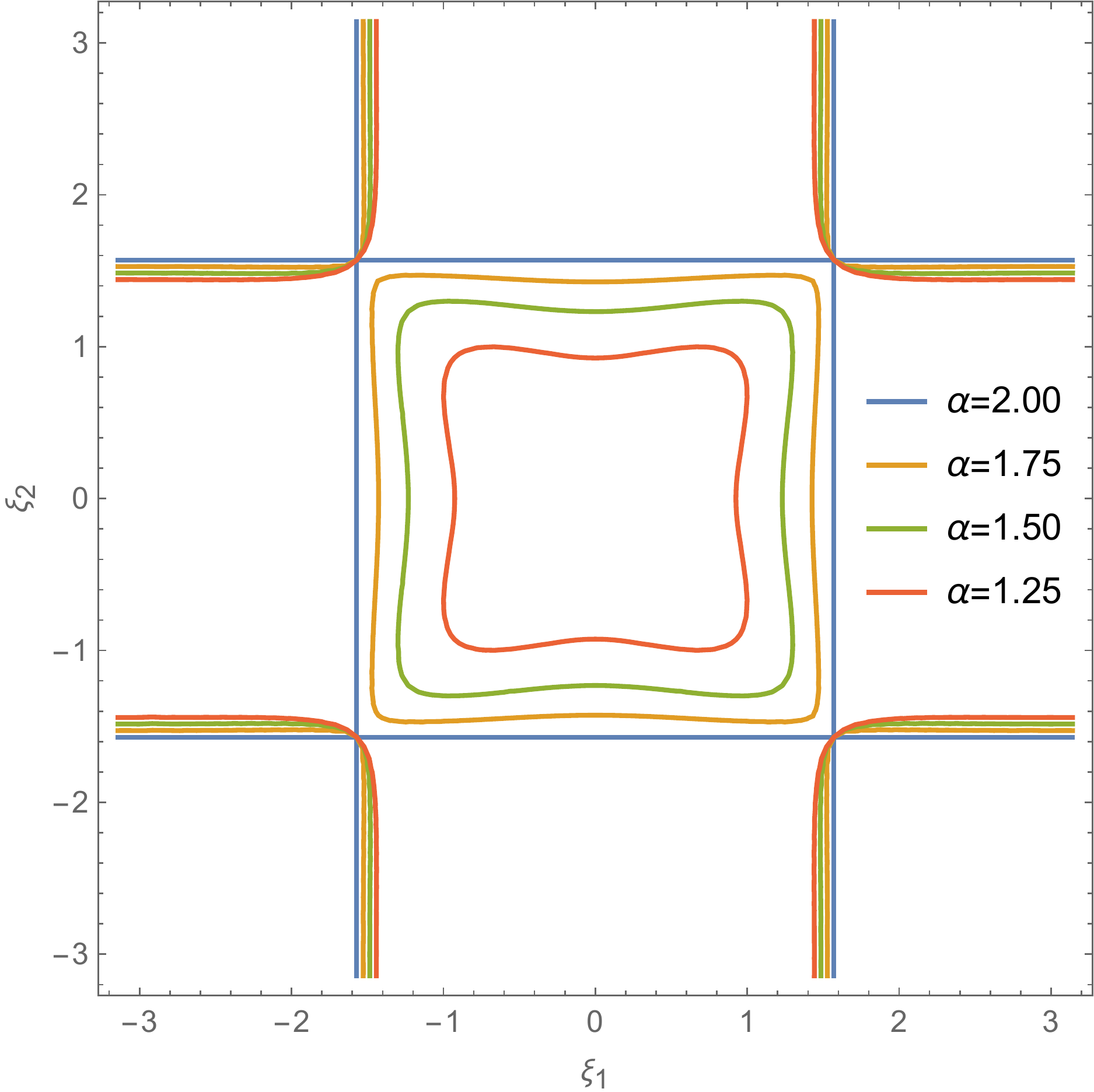}
\caption{The contour plots of $E_\alpha$ and its correspondence in $(\xi_1,\xi_2) \in M$ are given for different values of $\alpha$.}\label{contour}
\end{figure}

Let $\{e_1,e_2\}$ be the standard basis of $\mathbb{R}^2$, and with a slight abuse of notation, consider $\{e_1,e_2\}$ as a global orthonormal frame of $\mathbb{T}^2$. For each $\xi \in M$, let $\{k_1(\xi),k_2(\xi)\}$ be an orthonormal basis of $T_\xi M$, the tangent space at $\xi$, with the coordinate system $(x,y)$, i.e., for all $v \in T_\xi M$, there exists unique $(x,y) \in \mathbb{R}^2$ such that $v=xk_1(\xi)+yk_2(\xi)$. Moreover assume $\{k_1(\xi),k_2(\xi)\}$ diagonalizes $D^2w(\xi)$ as
\begin{equation*}
D^2w(\xi) = \begin{pmatrix}
k_1(\xi) & k_2(\xi)    
\end{pmatrix}
\begin{pmatrix}
\partial_{xx}w(\xi) & 0\\0& \partial_{yy}w(\xi)
\end{pmatrix}
\begin{pmatrix}
k_1(\xi) & k_2(\xi)    
\end{pmatrix}^{-1},
\end{equation*}
where $\partial_x = k_1(\xi) \cdot \nabla_\xi,\: \partial_y = k_2(\xi) \cdot \nabla_\xi$ are the directional derivatives. Suppose $\partial_{yy}w(\xi)=0$ for all $\xi \in E$, or equivalently, $k_2(\xi)$ is the direction along which $D^2w(\xi)$ is degenerate. Then it follows that $\partial_{xx}w(\xi)\neq 0$ since $D^2w(\xi)$ is not a zero matrix. Due to diagonalization, $\partial_{xy}w(\xi)=0$.

To further investigate the higher order directional derivatives, \cite[Lemma 5.4]{borovyk2017klein} is extended by induction and the product rule for derivatives whose proof is immediate and hence omitted.

\begin{lemma}\label{higher order}
For $m \geq 2$, let $f \in C^{m+1}(U)$ where $\xi_0 \in U \subseteq \mathbb{R}^d$. Suppose $D^2 f(\xi_0)$ has rank $d-1$ and let $k_d$ be a normalized eigenvector corresponding to eigenvalue zero. Suppose $(k_d(\xi_0)\cdot \nabla_\xi)^j f(\xi_0)=0$ for $2 \leq j \leq m$. Then $(k_d(\xi_0)\cdot \nabla_\xi)^{m+1} f(\xi_0)=0$ if and only if $(k_d(\xi_0)\cdot \nabla_\xi)^{m-1} \det D^2 f(\xi_0)=0$.
\end{lemma}

The inflection points of $\sin^2(\frac{\xi_i}{2})$ persist to exist as singular points of $D^2w(\xi)$ even when $\alpha<2$. By symmetry, the qualitative behavior of $J_{\Phi_v,\zeta}$ is the same near each point in $K_3$.  

\begin{lemma}\label{singularity2}
If $\xi\in E_\alpha \setminus K_3$, then $\partial_y^3 w(\xi) \neq 0$. Moreover $\partial_y^3 w(P) = 0$ for all $P \in K_3$.
\end{lemma}
\begin{proof}
For $\xi\in E_\alpha$, since $\nabla_\xi H(\xi,\alpha)= \Tilde{h}(\xi) \nabla_\xi h (\xi,\alpha)$ where $\Tilde{h}(\xi) \neq 0$, it suffices to show $v \cdot \nabla_\xi h(\xi,\alpha) = 0$ implies $\xi \in K_3$ where $v$ is any scalar multiple of $k_2(\xi)$. From \eqref{hessian}, let
\begin{equation}\label{eigenvector}
v = \begin{pmatrix}
-(2-\alpha)\sin \xi_1 \sin \xi_2\\
\alpha\cos^2 \xi_1 +2(\cos \xi_2-2)\cos \xi_1 + 2- \alpha
\end{pmatrix}
\:\text{or}\:
\begin{pmatrix}
\alpha\cos^2 \xi_2 +2(\cos \xi_1-2)\cos \xi_2 + 2- \alpha\\
-(2-\alpha)\sin \xi_1 \sin \xi_2
\end{pmatrix}
\end{equation}
Take the former; the proof for the latter is similar and thus is omitted. First assume $\sin \xi_2 \neq 0$.

In the $(a,b)$ coordinates, $\nabla_\xi h = -(\sin \xi_1 \partial_a h,\sin \xi_2 \partial_b h)$, and our task reduces to solving
\begin{equation}\label{singularity}
    \sin \xi_2\Big(-(2-\alpha)(1-a^2)\partial_a h +(\alpha a^2 + 2(b-2) a +2-\alpha)\partial_b h\Big)=0,
\end{equation}
by applying \Cref{higher order} with $m=2$. Modulo $\sin \xi_2$, \eqref{singularity} is a polynomial equation of degree $3$ in $b$ (or $a$) and therefore can be solved explicitly. The intersection of $h(a,b,\alpha)=0$ and \eqref{singularity} occurs at $(a,b)=(0,0)$.

If $\xi$ lies at the intersection of $\sin \xi_2=0$ and $h(\xi,\alpha)=0$, then it can be directly verified that the left vector of \eqref{eigenvector} is zero whereas the right vector is a scalar multiple of $\begin{pmatrix}
1\\0
\end{pmatrix}$. Then the claim can be proved similarly as before.
\end{proof}

The higher-order derivatives at critical points determine the height of the phase function.

\begin{lemma}\label{height}
For all $\xi \in M$,
\begin{equation*}
h(\Phi_{v_\xi}) =
\begin{cases} 
    1 &\mbox{if } \xi \notin E_\alpha,\\
    \frac{6}{5} & \mbox{if } \xi \in E_\alpha \setminus K_3,\\
    \frac{4}{3} &\mbox{if } \xi \in K_3.\end{cases}
\end{equation*}
\end{lemma}
\begin{proof}
Let $\xi \notin E_\alpha$. In any given coordinate system, say $(\Tilde{x},\Tilde{y})$, if $\partial_{\Tilde{x}\Tilde{y}}w(\xi) \neq 0$, then $d_{(\Tilde{x},\Tilde{y})}=1$. If $\partial_{\Tilde{x}\Tilde{y}}w(\xi) = 0$, then both $\partial_{\Tilde{x}}^2 w(\xi),\partial_{\Tilde{y}}^2 w(\xi) \neq 0$ due to non-degeneracy. In either case, $d_{(\Tilde{x},\Tilde{y})}=1$. Taking supremum over all such coordinate systems, the first claim has been shown. The rest follows similarly as \cite[Lemma 3.1]{borovyk2017klein} using \Cref{singularity2}. In particular, the Newton diagrams for $\xi \in E_\alpha \setminus K_3$ and $\xi \in K_3$ are given by
\begin{equation*}
\begin{split}
\mathcal{N}_d &= \{3x+2y=6: 0 \leq x \leq 2\},\\
\mathcal{N}_d &= \{2x+y=4: 1 \leq x \leq 2\},
\end{split}
\end{equation*}
respectively.
\end{proof}

The computation of heights depends only on the non-zero Taylor coefficients of $\Phi_v$, and therefore does not reflect the variations on $\alpha$. However the leading terms of asymptotics \eqref{asymptotics} depend on $\alpha$. Define
\begin{equation*}
I_{\Phi_v,\zeta}(\epsilon) = \int_{\{0<\Phi_v<\epsilon\}} \zeta(\xi)d\xi.
\end{equation*}
As in \eqref{asymptotics}, $I$ has an asymptotic expansion as $\epsilon \rightarrow 0+$,
\begin{equation*}
I_{\Phi_v,\zeta} \sim \sum_{j=0}^\infty (c_j(\zeta)+c_j^\prime(\zeta)\log \epsilon)\epsilon^{r_j},\:I_{-\Phi_v,\zeta} \sim \sum_{j=0}^\infty (C_j(\zeta)+C_j^\prime(\zeta)\log \epsilon)\epsilon^{r_j}, 
\end{equation*}
where $\{r_j\}$ is an increasing arithmetic sequence of positive rational numbers such that the (minimal) $r_0$ is determined only by the phase function that renders at least one of $c_0,c_0^\prime,C_0,C_0^\prime$ non-zero. For $0 \leq m \leq \infty$, let $-\frac{1}{m}$ be the slope of the subset of $\mathcal{N}_d$ that the bisectrix intersects. Define $\Phi_{v_\xi,pr}^+(x,\pm 1) = \Phi_{v_\xi,pr}(x,\pm 1)$ if $\Phi_{v_\xi,pr}(x,\pm 1)>0$ and zero otherwise. A summary of \cite[Theorem 1.1,1.2]{greenblatt2009asymptotic} that applies to our case is given. 
\begin{lemma}\label{greenblatt}
Let $\xi \in E_\alpha$. Then $\Phi_{v_\xi}=\Phi_{v_\xi}(x,y)$ in the coordinate system defined by $\{k_1(\xi),k_2(\xi)\}$ is superadapted. The slowest decay of the asymptotics is given by $\sigma_0 = r_0 = \frac{1}{h(\Phi_{v_\xi})}$. The leading terms have vanishing logarithmic terms, i.e., $c_0^\prime = C_0^\prime = 0$, and moreover
\begin{equation}\label{asymptotics6}
\begin{split}
c_0 &= \lim_{\epsilon \rightarrow 0} \frac{I_{\Phi_{v_\xi},\zeta}(\epsilon)}{\epsilon^{\sigma_0}} = \frac{\zeta(0,0)}{m+1} \int_\mathbb{R} \Phi_{v_\xi,pr}^+ (x,1)^{-\sigma_0} + \Phi_{v_\xi,pr}^+ (x,-1)^{-\sigma_0} dx\\
d_0 &=\lim_{\tau \rightarrow \infty} \frac{J_{\Phi_{v_\xi},\zeta}(\tau)}{\tau^{-\sigma_0}}=  \sigma_0 \Gamma(\sigma_0)\Big(e^{i\frac{\pi \sigma_0}{2}}c_0 + e^{-i\frac{\pi \sigma_0}{2 }}C_0\Big),    
\end{split}
\end{equation}
where $C_0$ is computed similarly by replacing $\Phi_{v_\xi}$ by its negative.
\end{lemma}
Since $\partial_y^3 w(P)=0$ for all $P\in K_3$ (see \Cref{singularity2}), the decay of $J$ is the slowest on $K_3$. Recalling that $K_3 \subseteq \bigcap\limits_{\alpha\in (1,2)}E_\alpha$, it is of interest to determine $d_0(\alpha)$ on $K_3$.

\begin{lemma}\label{cusp}
Let $\xi \in K_3$ and let $\zeta \in C^\infty_c$ be supported in a small neighborhood around $\xi$ in which $\xi$ is the unique critical point of $\Phi_{v_\xi}$. Then 
\begin{equation*}
    d_0(\alpha) =c\cdot \zeta(\xi)\alpha^{-\frac{3}{4}}(2-\alpha)^{-\frac{1}{4}},
\end{equation*}
where $c \in \mathbb{C}\setminus \{0\}$ is independent of $\zeta$ and $\alpha$.
\end{lemma}
\begin{proof}
Without loss of generality, let $\xi = (\frac{\pi}{2},\frac{\pi}{2})$. Define $k_1 = \frac{e_1(\xi)+e_2(\xi)}{\sqrt{2}},\: k_2 = \frac{-e_1(\xi)+e_2(\xi)}{\sqrt{2}}$ where the linear span of $\{k_1,k_2\}$ is coordinatized in $(x,y)$. Using $(\xi_1,\xi_2)$ as a coordinate for $\{e_1(0),e_2(0)\}$, we have
\begin{equation*}
    \xi_1 = \frac{\pi}{2} + \frac{x-y}{\sqrt{2}},\: \xi_2 = \frac{\pi}{2} + \frac{x+y}{\sqrt{2}}.
\end{equation*}
We first claim $(x,y)$ defines a superadapted system for $\Phi_{v_\xi}$ at $(x,y)=(0,0)$. Note that $\Phi_{v_\xi,pr} = -w_{pr}$ and
\begin{equation*}
w(x,y) = \Big\{\sin^2\Big(\frac{1}{2}(\frac{\pi}{2}+\frac{x-y}{\sqrt{2}})\Big)+\sin^2\Big(\frac{1}{2}(\frac{\pi}{2}+\frac{x+y}{\sqrt{2}})\Big)\Big\}^{\frac{\alpha}{2}}.
\end{equation*}
The second order derivatives are
\begin{equation}\label{second}
\partial_{xx}w(\xi)=-\frac{\alpha}{8}(2-\alpha),\: \partial_{xy}w(\xi)=\partial_{yy}w(\xi)=0,
\end{equation}
and the third order derivatives are
\begin{equation*}
\partial_{yyy}w(\xi) = \partial_{xxy}w(\xi)=0,\: \partial_{xyy}w(\xi)=-\frac{\alpha}{4\sqrt{2}},\: \partial_{xxx}w(\xi)= \frac{\alpha(\alpha^2-6\alpha+4)}{16\sqrt{2}}.
\end{equation*}
By direct computation, it can be verified that $\partial_y^j w(\xi)=0$ for all $j \geq 2$. These derivatives determine the Newton polyhedron, Newton diagram, and $w_{pr}$ given by
\begin{equation}\label{principal}
\begin{split}
\mathcal{N} &= \{2x+y \geq 4,\: x \geq 1,\: y \geq 0\},\: \mathcal{N}_d = \{2x+y=4,\: 1 \leq x \leq 2\},\\
\Phi_{v_\xi,pr} &= -w_{pr}(x,y)= \frac{\alpha(2-\alpha)}{16}x^2 + \frac{\alpha}{8\sqrt{2}}xy^2 =: Ax^2 + B xy^2.
\end{split}
\end{equation}
The bisectrix intersects $\mathcal{N}_d$ at $x=y=\frac{4}{3}=d(\Phi_{v_\xi})>1$. Since $\Phi_{v_\xi,pr}(x,\pm 1) = Ax^2+Bx$ has two real roots, $\{0,-\frac{B}{A}\}$, our first claim has been shown. It follows immediately from \Cref{greenblatt} that $r_0 = \sigma_0 = \frac{1}{d(\Phi_{v_\xi})} = \frac{3}{4}$. Moreover $c_0^\prime = C_0^\prime = 0$ and
\begin{equation}\label{asymptotics4}
\begin{split}
c_0 &= \frac{2}{3}\zeta(0,0) \int_\mathbb{R} \Phi_{v_\xi,pr}^+ (x,1)^{-\frac{3}{4}} + \Phi_{v_\xi,pr}^+ (x,-1)^{-\frac{3}{4}} dx\\
&= \frac{4}{3}\zeta(0,0)\Big(\int_{-\infty}^{-\frac{B}{A}}(Ax^2+Bx)^{-\frac{3}{4}} dx+\int_{0}^{\infty}(Ax^2+Bx)^{-\frac{3}{4}} dx\Big)= \frac{2^{\frac{27}{4}}\pi^{\frac{1}{2}}\Gamma(\frac{1}{4})}{3 \Gamma(\frac{3}{4})}\zeta(0,0)\alpha^{-\frac{3}{4}}(2-\alpha)^{-\frac{1}{4}},   
\end{split}
\end{equation}
by using \eqref{principal}; the computation of $C_0$, which amounts to replacing $\Phi_{v_\xi}$ by $-\Phi_{v_\xi}$, is similar and thus is omitted. The conclusion of lemma follows immediately from the explicit form of $d_0$ in \eqref{asymptotics6} and \eqref{asymptotics4}.
\end{proof}
\begin{remark}\label{classical}
If $\alpha=2$, $\partial_{xx}w(\xi)=0$ by \eqref{second}. Consequently, the quadratic term of $\Phi_{v_\xi,pr}$ vanishes and the Newton diagram is given by
\begin{equation*}
    \mathcal{N}_d = \{x+y=3,\: 1 \leq x \leq 3\}.
\end{equation*}
By analogy with the proof of \Cref{cusp}, it is expected that the reciprocal of the distance of this diagram, $\frac{2}{3}$, yields the sharp decay rate of the corresponding oscillatory integral. Indeed, this expectation coincides with the result obtained in \cite{stefanov2005asymptotic}.
\end{remark}
\begin{lemma}\label{fold}
For all $\xi \in E_\alpha \setminus K_3$, let $\zeta$ be as \Cref{cusp} such that its support does not intersect $K_3$. Then we have
\begin{equation}\label{asymptotics7}
    d_0 =c\cdot \zeta(\xi)|\Tilde{h}(\xi,\alpha)\partial_y h(\xi,\alpha)|^{-\frac{1}{3}}|Tr D^2w(\xi)|^{-\frac{1}{6}},
\end{equation}
where $c \in \mathbb{C}\setminus \{0\}$ is independent of $\zeta,\alpha,\xi$.
\end{lemma}
\begin{proof}
By taking $\partial_y = k_2(\xi)\cdot \nabla_\xi$ on $H$,
\begin{equation*}
\partial_y H(\xi,\alpha) = \Tilde{h}(\xi,\alpha)\partial_y h(\xi,\alpha) = \partial_x^2 w(\xi) \partial_y^3w (\xi),   
\end{equation*}
by $h(\xi,\alpha), \partial_y^2 w(\xi)=0$. Since the trace of a matrix is the sum of eigenvalues, $Tr D^2w(\xi) = \partial_x^2 w(\xi)$, and therefore
\begin{equation}\label{eq1}
\partial_y^3 w(\xi) = \frac{\Tilde{h}(\xi,\alpha)\partial_y h(\xi,\alpha)}{Tr D^2w(\xi)}.    
\end{equation}
By \Cref{height}, $\partial_x^2 w(\xi),\partial_y^3 w(\xi) \neq 0$, which yields $\sigma_0 = r_0= \frac{5}{6}$ and
\begin{equation*}
\begin{split}
w_{pr}(x,y) &= \frac{\partial_x^2w(\xi)}{2}x^2 + \frac{\partial_y^3 w(\xi)}{6}y^3\\
\pm \Phi_{v_\xi,pr}(x,\pm 1) &= \pm \frac{\partial_x^2w(\xi)}{2}x^2 \pm \frac{\partial_y^3 w(\xi)}{6},    
\end{split}
\end{equation*}
and therefore the coordinate system $(x,y)$ is superadapted. By \Cref{greenblatt}, it suffices to compute $c_0$, given by
\begin{equation*}
\begin{split}
c_0 &= c \cdot \zeta(\xi) \int_\mathbb{R} \Phi_{v_\xi,pr}^+ (x,1)^{-\frac{5}{6}}+\Phi_{v_\xi,pr}^+ (x,-1)^{-\frac{5}{6}}dx\\
&= c \cdot \zeta(\xi) |\partial_x^2 w(\xi)|^{-\frac{1}{2}}|\partial_y^3 w(\xi)|^{-\frac{1}{3}}\\
&= c\cdot \zeta(\xi) |\Tilde{h}(\xi,\alpha)\partial_y h(\xi,\alpha)|^{-\frac{1}{3}}|Tr D^2w(\xi)|^{-\frac{1}{6}},
\end{split}
\end{equation*}
where the last equation is by \eqref{eq1}.
\end{proof}
It is insightful to apply \eqref{asymptotics7} to obtain the series expansion of $d_0$. For $\{(a,b):h(a,b,\alpha)=0\}$, it suffices to consider $a \geq b$ or $a \leq b$ by the symmetry of $h$ under $(a,b) \mapsto (b,a)$. Define the two roots of $h(a,b,\alpha)=0$ in terms of $a,\alpha$ as
\begin{equation}\label{quadratic}
\begin{split}
B_P(a,\alpha) &= \frac{4 a-\alpha  a^2-(2-\alpha)+\sqrt{\left(\alpha  a^2-4 a+2-\alpha\right)^2-4 a^2 \alpha  (2 - \alpha )}}{2 a \alpha },\: a \in [-1,0)\\
B(a,\alpha) &= \frac{4 a-\alpha  a^2-(2-\alpha)-\sqrt{\left(\alpha  a^2-4 a+2-\alpha\right)^2-4 a^2 \alpha  (2 - \alpha )}}{2 a \alpha },\: a \in [\frac{2-\alpha}{\alpha},1].
\end{split}
\end{equation}
Several comments regarding $B_P,B$ are summarized below. The following lemma can be verified by direct computation using \eqref{quadratic}.
\begin{lemma}\label{quadratic2}
For all $\alpha \in (1,2)$, the curve $a\mapsto B_P(a,\alpha)$ parametrizes $\Gamma^1_{(a,b)}(\alpha) \cap \{a\leq b\}$ and $a\mapsto B(a,\alpha)$ parametrizes $\Gamma^2_{(a,b)}(\alpha) \cap \{a \geq b\}$. The pointwise convergence $\lim\limits_{a\rightarrow 0-}B_P(a,\alpha)=0,\: \lim\limits_{\alpha\rightarrow 2-}B_P(a,\alpha)=0$ holds. Furthermore $B(\cdot,\alpha)$ obtains the global maxima on $[\frac{2-\alpha}{\alpha},1]$ at the boundary where $B(\frac{2-\alpha}{\alpha},\alpha)=B(1,\alpha)=\frac{2-\alpha}{\alpha}$. The global minimum is obtained at $a_m=(\frac{2-\alpha}{\alpha})^{\frac{1}{2}}$ and $B(a_m,\alpha) \geq 1-(1+\sqrt{2})(\alpha-1)$. 
\end{lemma}
\begin{corollary}
Consider $d_0 = d_0(a,B_P(a,\alpha),\alpha,\zeta)$ defined on $\Gamma^1_{(a,b)}(\alpha) \cap \{a\leq b\}$ and let $\zeta \in C^\infty_c$ be supported in a small neighborhood around $(a,B_P(a,\alpha))$ excluding $(a,b)=(0,0)$. Then for some $c\in\mathbb{C}\setminus\{0\}$ independent of $a,\alpha,\zeta$,
\begin{equation}\label{asymptotics8}
    d_0 = c \cdot \zeta(a,B_P(a,\alpha))|a|^{-\frac{1}{3}}\sum_{j=0}^\infty a_j(\alpha) |a|^j
\end{equation}
holds where the series converges absolutely for all $a \in [-1,0)$. The coefficients $\{a_j\}$ can be computed explicitly; for example, $a_0(\alpha) = c^\prime (2-\alpha)^{-\frac{1}{6}}\alpha^{-\frac{5}{6}}$ where $c^\prime$ is a non-zero numerical constant independent of $\zeta,\alpha$.
\end{corollary}
\begin{proof}
The series expansion \eqref{asymptotics8} is shown by the general formula \eqref{asymptotics7}. The pointwise absolute convergence on $a\in [-1,0)$ follows from the analyticity of the RHS of \eqref{asymptotics7} on $\xi \in M$ corresponding to $\Gamma^1_{(a,b)}(\alpha) \cap \{a\leq b\}$.  
\end{proof}
\begin{remark}
By direct computation, $a_j(\alpha)$ contains the term $(2-\alpha)^{-p_j}$ for all $j \geq 0$ for some $p_j>0$, and therefore one obtains a singular behavior of the leading term of $J$ as $\alpha \rightarrow 2-$. Another interesting regime is when $\xi \rightarrow P$, or equivalently $a\rightarrow 0-$, along $E_\alpha$. A qualitative difference between a cusp ($\xi=P$) and a fold ($\xi \neq P$) is manifested quantitatively by the blow-up $|a|^{-\frac{1}{3}}$ as $a \rightarrow 0-$.
\end{remark}
On the other hand, consider the case $\alpha \rightarrow 1+$. By symmetry, consider $\Gamma^2_{(a,b)}(\alpha) \cap \{a \geq b\}$. The asymptotic behavior of $d_0(a(\alpha),B(a(\alpha),\alpha),\alpha)$ where $a(\alpha)\in [\frac{2-\alpha}{\alpha},1]$ is computed as $\alpha \rightarrow 1+$.

\begin{corollary}\label{wave limit}
Consider $d_0 = d_0(a(\alpha),B(a(\alpha),\alpha),\alpha,\zeta)$ and let $\zeta\in C^\infty_c$ be supported in a small neighborhood around $(a(\alpha),B(a(\alpha),\alpha))$. Then there exists $\alpha_0>1$ such that
\begin{equation*}
    |d_0| \simeq |\zeta(a,B(a,\alpha))| (\alpha-1)^{\frac{2}{3}-\frac{5\alpha}{12}},
\end{equation*}
whenever $\alpha \in (1,\alpha_0]$ and $a \in [\frac{2-\alpha}{\alpha},1]$. Furthermore the implicit constants depend only on $\alpha_0$.
\end{corollary}
\begin{proof}
Let $a=1-\Tilde{a},\: b=1-\Tilde{b}$. By $a \in [\frac{2-\alpha}{\alpha},1]$ and \Cref{quadratic2},
\begin{equation}\label{sma2}
    0 \leq \Tilde{a} \leq \frac{2}{\alpha}(\alpha-1),\: \frac{2}{\alpha}(\alpha-1) \leq \Tilde{b} \leq (1+\sqrt{2})(\alpha-1). 
\end{equation}
Observe that the image of $\Gamma^{2}_{(a,b)}$ under the inverse cosine lies in $[-\frac{\pi}{2},\frac{\pi}{2}]^2$, and it can be verified that for $z \in [0,\frac{\pi}{2}]$,
\begin{equation}\label{sma}
    \frac{z}{2} \leq \sin z \leq z ,\: 1 - \frac{z^2}{2}\leq \cos z \leq 1-\frac{z^2}{4},\: \sqrt{2}z^{\frac{1}{2}} \leq \cos^{-1}(1-z) \leq 2 z^{\frac{1}{2}}.
\end{equation}
Define $(\xi_1,\xi_2) \in M$ as the inverse cosine of $(a(\alpha),B(a(\alpha),\alpha))$, respectively; for definiteness, let $\xi_i \geq 0$. In the polar coordinate where $r^2 = \xi_1^2 + \xi_2^2$, \eqref{sma} is used to obtain
\begin{equation}\label{detail1}
|\Tilde{h}(\xi,\alpha)|\simeq \frac{\alpha^2}{w(\xi)^{\frac{8}{\alpha}-2}}(\cos\xi_1+\cos\xi_2-2)\simeq r^{-6+2\alpha},
\end{equation}
where we neglect any powers of $\alpha$ since they can be uniformly bounded for $\alpha \in (1,2)$. To estimate $Tr D^2w$, define $(\Tilde{r},\Tilde{\theta})$ such that $\Tilde{r}^2 = \Tilde{a}^2 + \Tilde{b}^2,\: \tan \Tilde{\theta} = \frac{\Tilde{b}}{\Tilde{a}}$. Then
\begin{equation*}
\begin{split}
|Tr D^2w| &\simeq r^{\alpha-4} |\alpha(a^2+b^2) + 2((b-2)a+(a-2)b) + 4 - 2\alpha|\\
&=r^{\alpha-4} |-2\alpha (\Tilde{a}+\Tilde{b})+\alpha(\Tilde{a}^2+\Tilde{b}^2)+4\Tilde{a}\Tilde{b}|
\end{split}
\end{equation*}
We claim
\begin{equation*}
-2\alpha (\Tilde{a}+\Tilde{b}) \leq -2\alpha (\Tilde{a}+\Tilde{b})+\alpha(\Tilde{a}^2+\Tilde{b}^2)+4\Tilde{a}\Tilde{b} \leq -\alpha(\Tilde{a}+\Tilde{b}).    
\end{equation*}
The lower bound is trivial since $\Tilde{a},\Tilde{b}\geq 0$. The upper bound is equivalent to
\begin{equation*}
    \frac{\Tilde{a}^2+\Tilde{b}^2}{\Tilde{a}+\Tilde{b}} = \frac{\Tilde{r}}{\sin\Tilde{\theta}+\cos\Tilde{\theta}}\leq \frac{\alpha}{\alpha+2},
\end{equation*}
which holds uniformly on $\alpha \in (1,2)$ if $\Tilde{r} \leq \frac{1}{3}$. Hence for all $\Tilde{a},\Tilde{b}$ sufficiently small,
\begin{equation}\label{detail2}
|TrD^2w| \simeq r^{\alpha-4} (\Tilde{a}+\Tilde{b}) \simeq r^{\alpha-4}(\alpha-1),   
\end{equation}
by \eqref{sma2}. Likewise for sufficiently small $\Tilde{a},\Tilde{b}$
\begin{equation}\label{detail3}
|\partial_y h| \simeq (\alpha-1)\Tilde{b}^{\frac{1}{2}} \simeq (\alpha-1)^{\frac{3}{2}}.
\end{equation}
For all $\epsilon>0$, since $\Gamma^2_{(a,b)}(\alpha) \subseteq \{(a,b) \in [1-\epsilon,1]\times [1-\epsilon,1]\}$ whenever $\alpha \in (1,\alpha_0(\epsilon)]$ for some $\alpha_0(\epsilon)>1$, there exists $\alpha_0>1$ sufficiently close to $1$ such that all small angle approximations are justified (see \eqref{sma}) and
\begin{equation*}
\begin{split}
\frac{|d_0|}{c\cdot \zeta(a(\alpha),B(a(\alpha),\alpha))} &= |\Tilde{h}(\xi_1,\xi_2,\alpha)\partial_y h(\xi_1,\xi_2,\alpha)|^{-\frac{1}{3}}|Tr D^2w(\xi_1,\xi_2)|^{-\frac{1}{6}}\\
&\simeq (\alpha-1)^{-\frac{2}{3}}r^{\frac{8}{3}-\frac{5\alpha}{6}}.
\end{split}
\end{equation*}
by \eqref{asymptotics7}, \eqref{detail1}, \eqref{detail2}, and \eqref{detail3}. Combining with
\begin{equation*}
r\simeq |\xi_1|+|\xi_2| = \cos^{-1}(1-\Tilde{a}) + \cos^{-1}(1-\Tilde{b})\simeq (\alpha-1)^{\frac{1}{2}},    
\end{equation*}
the proof is complete.
\end{proof}
\begin{remark}
As can be seen in \Cref{contour}, the trajectory $\alpha \mapsto (\frac{2-\alpha}{\alpha},\frac{2-\alpha}{\alpha})$ traces the intersection of $\Gamma^2_{(a,b)}$ and the bisectrix. For $a(\alpha) = \frac{2-\alpha}{\alpha}$, an explicit computation yields
\begin{equation*}
    d_0 =c\cdot\zeta(a(\alpha),a(\alpha)) 2^{-\frac{5\alpha}{12}}\alpha^{-(\frac{7}{6}-\frac{5\alpha}{12})}(\alpha-1)^{\frac{2}{3}-\frac{5\alpha}{12}}(2-\alpha)^{-\frac{1}{2}},
\end{equation*}
for some $c \in \mathbb{C}\setminus\{0\}$ independent of $\alpha,\zeta$. Suppose $supp(\zeta)$ is sufficiently small such that $\zeta(a(\alpha),a(\alpha))=1$. Then note that $d_0 \xrightarrow[\alpha\rightarrow 2-]{}\infty$ as $(\frac{2-\alpha}{\alpha},\frac{2-\alpha}{\alpha})$ approaches the origin, which corresponds to the cusp $K_3$. Furthermore $d_0 \xrightarrow[\alpha\rightarrow 1+]{}0$ as $(\frac{2-\alpha}{\alpha},\frac{2-\alpha}{\alpha})$ approaches $(a,b)=(1,1)$, which corresponds to the origin of $\mathbb{T}^2$ where $w(\xi)$ blows up. 

Another example of trajectory, given by $(a,b)=(1,\frac{2-\alpha}{\alpha})$ with the leading term
\begin{equation}\label{k2}
    d_0 = c\cdot \zeta(1,\frac{2-\alpha}{\alpha})\alpha^{-\frac{5}{12}(4-\alpha)} (\alpha-1)^{\frac{2}{3}-\frac{5\alpha}{12}},
\end{equation}
shows a different qualitative behavior as $\alpha\rightarrow 2-$.
\end{remark}
\begin{proof}[Proof of \Cref{conjecture}]
For all $\tau>0$,
\begin{equation*}
    \sup_{v\in \mathbb{R}^2}|J_{\Phi_v,\eta(\frac{\cdot}{N})}| \leq \| \eta \|_{L^1(\mathbb{T}^2)}N^2,
\end{equation*}
by the triangle inequality. Hence $\tau\geq 1$.

Considering $supp \Big(\eta(\frac{\cdot}{N})\Big) = \{|\xi| \in [\frac{\pi}{2}N,2\pi N]\}$ and
\begin{equation*}
\min\limits_{\xi\in K_2\cup K_3}|\xi| = \cos^{-1}(\frac{2-\alpha}{\alpha}),\: \max\limits_{\xi\in \Gamma^2_{(\xi_1,\xi_2)}(\alpha)}|\xi| = \sqrt{2}\cos^{-1}(\frac{2-\alpha}{\alpha})    
\end{equation*}
obtained at 
\begin{equation*}
\{(\xi_1,\xi_2):(\pm \cos^{-1}(\frac{2-\alpha}{\alpha}),0),(0,\pm \cos^{-1}(\frac{2-\alpha}{\alpha}))\},\:\{(\xi_1,\xi_2):\pm \cos^{-1}(\frac{2-\alpha}{\alpha}),\pm \cos^{-1}(\frac{2-\alpha}{\alpha})\}\subseteq M,    
\end{equation*}
respectively, define $N_\alpha \in 2^\mathbb{Z}$ to be the largest number satisfying
\begin{equation*}
2\pi N_\alpha < r_\alpha:=\cos^{-1}(\frac{2-\alpha}{\alpha}).    
\end{equation*}
Note that $N_\alpha$ increases as $\alpha$ increases with $\lim\limits_{\alpha\rightarrow 1+}N_\alpha=0$ and $N_2 = 2^{-3}$. Using the support condition of $\eta(\frac{\cdot}{N})$, the set of $N\in 2^\mathbb{Z}$ that satisfies the RHS of \eqref{prove} is given by
\begin{equation}\label{N}
\begin{split}
supp(\eta(\frac{\cdot}{N}))\cap K_3 \neq \emptyset &\Longleftrightarrow N \in S_3:=\{2^0,2^{-1}\}\\
supp(\eta(\frac{\cdot}{N}))\cap (K_2\setminus K_3) \neq \emptyset &\Longleftrightarrow N \in S_2:=[\frac{1}{2\pi}r_\alpha,\frac{2\sqrt{2}}{\pi}r_\alpha]\setminus S_3\\
supp(\eta(\frac{\cdot}{N}))\cap (K_1\setminus (K_2 \cup K_3)) \neq \emptyset &\Longleftrightarrow N \in S_1:=\Big((\frac{2\sqrt{2}}{\pi}r_\alpha,2^{-2}] \cup (0,N_\alpha]\Big)\setminus (S_2 \cup S_3).
\end{split}
\end{equation}

Suppose $N >N_\alpha$. For every $\xi \in \mathbb{T}^2 \setminus B(0,\frac{r_\alpha}{2})$, there exists a neighborhood $\Omega_\xi(\alpha)$ containing $\xi$ and a constant $C_\xi(\alpha)>0$ such that for all $\zeta \in C^\infty_c(\Omega_\xi)$,
\begin{equation}\label{ik}
\sup\limits_{v\in \mathbb{R}^2}|J_{\Phi_v,\zeta}| \leq C_\xi \| \zeta \|_{C^3(\mathbb{R}^2)}\tau^{-\frac{1}{h(\Phi_{v_\xi})}},    
\end{equation}
by \cite[Theorem 1.1]{ikromov2011uniform}. By compactness, an open cover $\{\Omega_\xi\}$ of $\mathbb{T}^2\setminus B(0,\frac{r_\alpha}{2})$ reduces to a finite subcover $\{\Omega_{\xi_j}\}_{j=1}^{n_0}$, where $n_0 = n_0(\alpha)\in \mathbb{N}$, and let $\{\phi_j\}_{j=1}^{n_0}$ be a ($\alpha$-dependent) partition of unity  subordinate to the finite sub-cover. Note that if $\xi \in K_1$, then $U_\xi \cap K_2 = \emptyset$, and if $\xi \in K_2$, then $U_\xi \cap K_3 = \emptyset$ since the oscillatory indices of $\Phi_{v_\xi}$ at $\xi$ are distinct on each $K_j$. Consequently each $\xi \in K_3$ contributes to the finite sub-cover for all $\alpha \in (1,2)$.

Let $\eta_{N,j}(\cdot) = \eta(\frac{\cdot}{N})\phi_j(\cdot)$. Then $\|\eta_{N,j}\|_{C^3}\lesssim N^{-3}$ where the implicit constant depends on the given partition of unity. Since $N>N_\alpha$, we have $N^{-3} \leq C(\alpha)N^{2-\frac{3}{4}\alpha}$ where $C(\alpha) = N_\alpha^{-5+\frac{3}{4}\alpha}$.

By \Cref{height} and $\tau \geq 1$, the slowest decay occurs on $K_3$ with $h(\Phi_v) = \frac{4}{3}$. By \Cref{height}, \eqref{ik}, and the triangle inequality, 
\begin{equation}\label{dispersive est3}
    \sup_{v\in\mathbb{R}^2}|J_{\Phi_v,\eta(\frac{\cdot}{N})}|\leq \sup_{v\in\mathbb{R}^2}\sum_{j=1}^{n_0} |J_{\Phi_v,\eta_{N,j}}| \leq \Big(\sum_{j=1}^{n_0} C_{\xi_j}\| \eta_{N,j} \|_{C^3}\Big) \tau^{-\frac{3}{4}} \lesssim_\alpha \sum_{j=1}^{n_0} C_{\xi_j}\cdot N^{2-\frac{3}{4}\alpha}\tau^{-\frac{3}{4}}.
\end{equation}
For $N>N_\alpha$, a similar argument using the partition of unity and \eqref{ik} yields \eqref{prove} with sharp decay rates $\sigma_0 \in \{\frac{3}{4},\frac{5}{6},1\}$.

Suppose $N \leq N_\alpha \leq \frac{1}{8}$. Recalling that $v=\frac{x}{h\tau}$, do a change of variable $\xi \mapsto N \xi$ to obtain
\begin{equation*}
    \sup_{v\in \mathbb{R}^2}|J_{\Phi_v,\eta(\frac{\cdot}{N})}(\tau)| = N^2\sup_{x\in \mathbb{R}^2} \left|\int_{\mathbb{R}^2}e^{i(x\cdot \xi - \tau w(N\xi))}\eta(\xi)d\xi\right|.
\end{equation*}

By adopting the proof of \cite[Proposition 1]{cho2011remarks}, we obtain sharp dispersive estimates of a free solution governed by a non-smooth, non-homogeneous dispersion relation. A change of variables
\begin{equation*}
    z_i = \frac{2}{N} \sin (\frac{N\xi_i}{2}),\: \xi_i = \frac{2}{N}\sin^{-1}(\frac{Nz_i}{2}),\: \tau \mapsto 2^\alpha \tau,
\end{equation*}
with $J_c(z) = \Big((1-(\frac{Nz_1}{2})^2)(1-(\frac{Nz_2}{2})^2)\Big)^{-\frac{1}{2}}$, yields
\begin{equation*}
\begin{split}
    N^2\int_{\mathbb{R}^2}e^{i(x\cdot \xi - \tau w(N\xi))}\eta(\xi)d\xi&= N^2 \int_{\mathbb{R}^2}e^{i(x\cdot\xi(z)-\tau N^\alpha\rho^\alpha)}\eta(\xi(z))J_c(z)dz\\
    &= N^2\int_0^\infty e^{-i\tau N^\alpha \rho^\alpha}G(\rho,x,N)\rho d\rho:=I,
\end{split}
\end{equation*}
where we denote $(r,\theta)$ and $(\rho,\phi)$ as the polar coordinates for $x=(x_1,x_2)$ and $z=(z_1,z_2)$, respectively, and
\begin{equation*}
\begin{split}
G(\rho,x,N) = G(\rho) &= \int_{S^1} e^{i x \cdot \xi(z)}\eta(\xi(z))J_c(z)d\phi(z)\\
&=\int_0^{2\pi} e^{i\lambda \Phi_G(\phi)} \eta(\xi(\rho,\phi))J_c(z(\rho,\phi))d\phi,
\end{split}
\end{equation*}
where $\lambda = \rho r$ and 
\begin{equation}\label{phase}
\Phi_G(\phi) = \frac{2}{\rho N}\Big(\cos\theta \sin^{-1}(\frac{N\rho \cos\phi}{2})+\sin\theta \sin^{-1}(\frac{N\rho \sin\phi}{2})\Big).
\end{equation}

By the support condition of $\eta$,
\begin{equation*}
    (\frac{N\pi}{4})^2 \leq \sin^{-1}(\frac{Nz_1}{2})^2+\sin^{-1}(\frac{Nz_2}{2})^2 \leq (N\pi)^2,
\end{equation*}
and the small angle approximation $z\leq \sin^{-1}z \leq 2z$ on $z\in [0,\frac{1}{\sqrt{2}}]$, one obtains
\begin{equation*}
    \frac{\pi}{4} \leq \rho \leq 2\pi,
\end{equation*}
since $\frac{N |z_i|}{2} = |\sin(\frac{N \xi_i}{2})| \leq \frac{1}{\sqrt{2}}$. When clear in context, we use the same symbols $\eta,J_c$ for the representations in different variables. We prove
\begin{equation}\label{dispersive est}
    |I| \lesssim (\alpha-1)^{-\frac{1}{2}}N^{2-\alpha}\tau^{-1},
\end{equation}
from which the proof is complete by interpolating with the trivial bound $|I| \lesssim N^2$.

Let $r_0>0$, independent of $N$, to be specified later and suppose $r \leq r_0$. Integration by parts yields
\begin{equation}\label{ibp}
    I = \frac{N^{2-\alpha}}{i\alpha \tau}\int e^{-i\tau N^\alpha \rho^\alpha}\partial_\rho (G(\rho) \rho^{2-\alpha})d\rho.
\end{equation}
Since $|G|\lesssim 1$ and the domain of integration is supported away from the origin,
\begin{equation*}
    |G(\rho)\partial_\rho (\rho^{2-\alpha})|\lesssim 1.
\end{equation*}
By the chain rule $\partial_\rho = \cos \phi \partial_{z_1}+\sin \phi \partial_{z_2}$ and the estimate,
\begin{equation*}
    |\partial_{z_i} e^{i x \cdot \xi(z)}| = |\partial_{z_i}e^{i\frac{2x_i}{N}\sin^{-1}(\frac{Nz_i}{2})}| = \frac{|x_i|}{\sqrt{1-(\frac{Nz_i}{2})^2}}\leq \sqrt{2}|x_i|\leq \sqrt{2}r_0,
\end{equation*}
one obtains 
\begin{equation*}
    \sup_{N\leq N_\alpha}\left|\int_{S^1} \partial_\rho \Big(e^{i x \cdot \xi(z)}\Big)\eta(\xi(z))J_c(z)d\phi(z)\right|\lesssim r_0.
\end{equation*}
By repeated applications of the product and chain rule,
\begin{equation*}
\sup_{N \leq N_\alpha}|\partial_\rho^k \eta(\xi(z))| \lesssim_{k,\eta} 1,\:\sup_{N\leq N_\alpha}|\partial_\rho^k J_c(z)| \lesssim_{k} 1,
\end{equation*}
for all $k \geq 0$ and therefore
\begin{equation*}
    |\partial_\rho(G(\rho)\rho^{2-\alpha})| \lesssim 1+r_0,
\end{equation*}
altogether implying
\begin{equation*}
|I| \lesssim_{r_0,\alpha} N^{2-\alpha}\tau^{-1}.
\end{equation*}

Suppose $r>r_0$. From \eqref{phase},
\begin{equation*}
\begin{split}
\partial_\phi \Phi_G(\rho,\phi) &= - \frac{\cos\theta \sin\phi}{(1-(\frac{N\rho \cos\phi}{2})^2)^{\frac{1}{2}}}+ \frac{\sin\theta \cos\phi}{(1-(\frac{N\rho \sin\phi}{2})^2)^{\frac{1}{2}}}\\
\partial_\phi^2 \Phi_G(\rho,\phi)&=-(1-(\frac{N\rho}{2})^2)\Big(\frac{\cos\theta \cos\phi}{(1-(\frac{N\rho \cos\phi}{2})^2)^{\frac{3}{2}}}+ \frac{\sin\theta \sin\phi}{(1-(\frac{N\rho \sin\phi}{2})^2)^{\frac{3}{2}}}\Big).
\end{split}
\end{equation*}
For a fixed $\rho>0$, denote $\Phi_G(\phi) = \Phi_G(\rho,\phi)$. The critical points correspond to the roots of $\partial_\phi \Phi_G$, and are the solutions to
\begin{equation}\label{critical point}
g(\rho,\phi)\tan\phi = \tan\theta,\:g(\rho,\phi) :=\Big(\frac{1-(\frac{N\rho \sin\phi}{2})^2}{1-(\frac{N\rho \cos\phi}{2})^2}\Big)^{\frac{1}{2}}.
\end{equation}
By direct computation, $g(\rho,\cdot)$ is a strictly positive $\pi$-periodic even function such that $g(\rho,\frac{\pi}{4}+\frac{k\pi}{2})=1$ for all $k \in \mathbb{Z}$. Furthermore $g(\rho,\phi)>1$ if $\phi \in [0,\frac{\pi}{4})$, $g(\rho,\phi)<1$ if $\phi \in (\frac{\pi}{4},\frac{\pi}{2}]$, and
\begin{equation*}
    \partial_{\phi}|_{\phi = \frac{\pi}{4}} \Big(g(\rho,\cdot)\tan(\cdot)\Big) = 4 - \frac{16}{8-N^2\rho^2}>0. 
\end{equation*}
Therefore $\tan\phi < g(\rho,\phi)\tan \phi <1$ on $(0,\frac{\pi}{4})$ and $g(\rho,\phi)\tan \phi < \tan\phi$ on $(\frac{\pi}{4},\frac{\pi}{2})$. By symmetry, suppose $x_i \geq 0$ for $i=1,2$, and therefore $0\leq \tan\theta \leq \infty$. By graphing $\phi \mapsto g(\rho,\phi)\tan\phi$ on $[0,2\pi)$, there exist two solutions, $\phi_{\pm}(\rho)$, where $\phi_+ \in [0,\frac{\pi}{2}]$ and $\phi_{-} = \phi_+ + \pi$. A crude estimate $|\phi_{\pm} - \theta_{\pm}|\leq \frac{\pi}{4}$ follows from a further inspection of the graph where $\theta_+ = \theta,\: \theta_{-}=\theta+\pi$.

The critical points are non-degenerate. If $\phi$ satisfies $\partial_\phi^2 \Phi_G=0$, then $g(\rho,\phi)^3 \cot \phi = -\tan\theta$. 
Assuming $\partial_{\phi}^2 \Phi_G(\phi_{\pm})=0$ and substituting \eqref{critical point}, we have $g(\rho,\phi_{\pm})^2 + \tan^2\phi_{\pm}=0$, a contradiction. Since $\phi_{\pm}$ and $\theta_{\pm}$ are in the same quadrants, we have $\cos\theta_{\pm}\cos\phi_{\pm},\:\sin\theta_{\pm}\sin\phi_{\pm}\geq 0$, and therefore
\begin{equation}\label{phase3}
|\partial_{\phi}^2 \Phi_G(\phi_{\pm})| \simeq \cos (\phi_{\pm} - \theta_{\pm})\simeq 1,
\end{equation}
independent of $N,\rho$.

Construct $\chi_{\pm} \in C^\infty_c(\mathbb{R}_\phi)$ given by $\chi_{+} = 1$ on $[0,\frac{\pi}{2}]$, supported in $(-\frac{\pi}{8},\frac{5\pi}{8})$, $\chi_{-} = 1$ on $[\pi,\frac{3\pi}{2}]$, supported in $(\frac{7\pi}{8},\frac{13\pi}{8})$, and define $\chi_0 = 1-(\chi_{+}+\chi_{-})$. Since $|\phi_{\pm} - \theta_{\pm}|\leq \frac{\pi}{4}$, we have $\chi_{\pm}(\phi_{\pm})=1$ for all $N \leq N_\alpha$. Let $\Tilde{\chi}_{\pm}:=\eta J_c  \chi_{\pm}$ and $\Tilde{\chi}_0 := \eta J_c  \chi_{0}$. Note that since
\begin{equation}\label{bump}
|\partial_{\phi}^k \eta|, |\partial_{\phi}^k J_c|, |\partial_{\phi}^k \chi_{\pm}| \lesssim_k 1,    
\end{equation}
for all $k\geq 0$ independent of $N,\rho$, so are the higher order partial derivatives (in $\phi$) of $\Tilde{\chi}_{\pm}$. Define
\begin{equation*}
    G_{\pm}(\lambda) = \int_0^{2\pi}e^{i\lambda \Phi_G(\phi)}\Tilde{\chi}_{\pm}(\phi)d\phi,
\end{equation*}
and similarly for $G_0$, and hence $G = G_+ + G_{-}+ G_0$. By \cite[Chapter VIII, Proposition 3]{stein1993harmonic}, $G_{\pm}$ has the asymptotics as $\lambda\rightarrow\infty$ given by
\begin{equation}\label{stationary2}
    G_{\pm}(\lambda) = \sqrt{\frac{2\pi}{|\partial_{\phi}^2 \Phi_G(\phi_{\pm})|}} e^{i(\lambda \Phi_G(\phi_{\pm})-\frac{\pi}{4})}\Tilde{\chi}_{\pm}(\phi_{\pm})\lambda^{-\frac{1}{2}} + \Tilde{G}_{\pm}(\lambda).
\end{equation}
More precisely, for all $k \in \mathbb{N}\cup\{0\}$, there exists $\lambda_0(k),\:C(k)>0$ such that
\begin{equation}\label{rapid decay2}
    |\partial_\lambda^k \Tilde{G}_{\pm}| \leq C\lambda^{-(\frac{3}{2}+k)},
\end{equation}
for all $\lambda \geq \lambda_0$. Since the estimates \eqref{phase3}, \eqref{bump} are uniform with respect to $N,\rho$, the constants $\lambda_0,C$ can be chosen to be independent of $N,\rho$. Since $\rho \leq 2\pi$, let $r_0 \geq \frac{\max(\lambda_0(0),\lambda_0(1))}{2\pi}$.

Away from the critical points, the integral in $\phi$ yields a rapid decay in $\lambda$. We claim
\begin{equation}\label{rapid decay}
    |\partial_\lambda G_0| \lesssim_k \lambda^{-k},
\end{equation}
for all $\lambda>0$ and $k\geq 1$ uniformly in $N,\rho$. Since
\begin{equation*}
    \partial_{\lambda} G_0 = i \int_0^{2\pi} e^{i\lambda \Phi_G(\phi)}\Phi_G(\phi)\Tilde{\chi}_{0}(\phi)d\phi= -\frac{1}{\lambda} \int_0^{2\pi} e^{i\lambda \Phi_G(\phi)}\partial_\phi \Big(\frac{\Phi_G \Tilde{\chi}_{0}}{\partial_\phi \Phi_G}\Big)d\phi,
\end{equation*}
and $|\partial_\phi \Phi_G|\geq |\sin (\phi-\theta)| \geq \sin(\frac{\pi}{8})$ for all $\phi \in [\frac{5\pi}{8},\frac{7\pi}{8}]\cup [\frac{13\pi}{8},\frac{15\pi}{8}]$ and $\theta \in [0,\frac{\pi}{2}] \cup [\pi,\frac{3\pi}{2}]$, \eqref{rapid decay} is shown for $k=1$ by the triangle inequality; for $k\geq 2$, \eqref{rapid decay} is shown by repeated use of integration by parts.

For $r\geq r_0$, consider the integral in \eqref{ibp} with $G$ replaced by $G_0$. Since by \eqref{rapid decay},
\begin{equation*}
|\partial_\rho G_0 |= r |\partial_\lambda G_0| \lesssim \frac{r}{\lambda} = \frac{1}{\rho},  
\end{equation*}
the integration by parts yields an estimate consistent with \eqref{dispersive est}. By replacing $G$ by $\Tilde{G}_{\pm}$ in the same integral, the bound \eqref{dispersive est} follows by \eqref{rapid decay2}.

It remains to show that $I$ with $G$ replaced by the leading term of $G_+$ in \eqref{stationary2} satisfies \eqref{dispersive est}; the analysis on $G_{-}$ is similar and therefore is omitted. Consider
\begin{equation*}
II := N^2 r^{-\frac{1}{2}}\int_0^\infty e^{i\tau \Psi(\rho)}a(\rho)d\rho
\end{equation*}
where
\begin{equation*}
\begin{split}
a(\rho)&:= \eta(\rho,\phi_+)J_c(\rho,\phi_+)\chi_+(\phi_+)|\partial_\phi^2 \Phi_G(\phi_+)|^{-\frac{1}{2}}\rho^{-\frac{1}{2}}\\
\Psi(\rho) &= -N^\alpha \rho^\alpha + \Phi_G(\phi_+)\frac{r \rho}{\tau}.
\end{split}
\end{equation*}

The region of integration
\begin{equation*}
\{\rho\in [\frac{\pi}{4},2\pi]:\alpha N^\alpha \rho^{\alpha-1}<\frac{1}{2}|\partial_\rho(\Phi_G(\phi_+)\rho)|\frac{r}{\tau}\:\text{or}\:\alpha N^\alpha \rho^{\alpha-1}>2|\partial_\rho(\Phi_G(\phi_+)\rho)|\frac{r}{\tau}\}    
\end{equation*}
is included in the case
\begin{equation}\label{ibp2}
\frac{r}{\tau} \gg N|\nabla w(N\xi)|\:\text{or}\:\frac{r}{\tau}\ll N|\nabla w (N\xi)|   
\end{equation}
since $|\partial_\rho(\Phi_G(\phi_+)\rho)|\simeq 1$ uniformly in $N,\rho$ as can be observed in \eqref{vdc2}. For $r,\xi$ satisfying \eqref{ibp2}, the lower bound of the phase function on $supp(\eta)\subseteq \mathbb{R}^2_{\xi}$ is
\begin{equation*}
    |\nabla_\xi (\frac{x}{\tau}\cdot \xi - w(N\xi))| \geq \frac{N}{2}|\nabla w(N\xi)| \simeq_\alpha N^\alpha.
\end{equation*}
Let $E_i = supp(\eta) \cap \{|\partial_{\xi_i}(\frac{x}{\tau}\cdot \xi - w(N\xi))|\gtrsim N^\alpha\}$ for $i=1,2$. By direct computation,
\begin{equation*}
|\partial_{\xi_i}^2 (\frac{x}{\tau}\cdot \xi - w(N\xi))| \lesssim N^\alpha,
\end{equation*}
and thus by integration by parts,
\begin{equation*}
\begin{split}
N^2 \left|\int_{\mathbb{R}^2}e^{i(x\cdot \xi - \tau w(N\xi))}\eta(\xi)d\xi\right|&\leq N^2\Big(\left|\int_{E_1}e^{i(x\cdot \xi - \tau w(N\xi))}\eta(\xi)d\xi\right|+\left|\int_{E_2}e^{i(x\cdot \xi - \tau w(N\xi))}\eta(\xi)d\xi\right|\Big)\\
&\lesssim N^{2-\alpha}\tau^{-1}.
\end{split}
\end{equation*}
It suffices to assume $\frac{r}{\tau} \simeq N |\nabla w(N\xi)|\simeq N^\alpha(\xi_1^2+\xi_2^2)^{\frac{\alpha-1}{2}}\simeq N^\alpha$. By direct computation,
\begin{equation}\label{vdc2}
\begin{split}
\partial_\rho(\Phi_G(\phi_+)\rho) &= \frac{\cos\theta\cos\phi_+}{(1-(\frac{N\rho \cos\phi_+}{2})^2)^{\frac{1}{2}}}+\frac{\sin\theta\sin\phi_+}{(1-(\frac{N\rho \sin\phi_+}{2})^2)^{\frac{1}{2}}}\\
\partial_{\rho}^2(\Phi_G(\phi_+)\rho) &= \frac{N^2\rho}{4}\Big(\frac{\cos\theta \cos^3\phi_+}{(1-(\frac{N\rho \cos\phi_+}{2})^2)^{\frac{3}{2}}}+\frac{\sin\theta \sin^3\phi_+}{(1-(\frac{N\rho \sin\phi_+}{2})^2)^{\frac{3}{2}}}\Big)\\
&-\Big(\frac{\cos\theta \sin\phi_+}{(1-(\frac{N\rho \cos\phi_+}{2})^2)^{\frac{3}{2}}}-\frac{\sin\theta \cos\phi_+}{(1-(\frac{N\rho \sin\phi_+}{2})^2)^{\frac{3}{2}}}\Big)\partial_\rho \phi_+=: III+IV.
\end{split}
\end{equation}
We claim $|\partial_\rho^2(\Phi_G(\phi_+)\rho)|\lesssim N^2$. Since $N\rho \leq \frac{\pi}{4}$ and $\theta,\phi_+ \in [0,\frac{\pi}{2}]$,
\begin{equation*}
    \sup_{\rho \in [\frac{\pi}{4},2\pi]}|III| \lesssim N^2.
\end{equation*}
Since $|IV| \lesssim |\partial_\rho \phi_+|$, it suffices to show
\begin{equation}\label{phip}
    \sup_{\rho \in [\frac{\pi}{4},2\pi]}|\partial_\rho \phi_+| \lesssim N^2,
\end{equation}
which follows from implicitly differentiating \eqref{critical point}, thereby obtaining
\begin{equation*}
\begin{split}
\partial_\rho \phi_+(\rho) = -\frac{\partial_\rho g(\rho,\phi_+) \cdot \tan \phi_+}{\partial_\phi g(\rho,\phi_+) \cdot \tan \phi_+ + g \sec^2 \phi_+}= \frac{N^2 \rho \sin (4\phi_+)}{(1-(\frac{N\rho}{2})^2)(16-N^2\rho^2(1-\cos (4\phi)))}. 
\end{split}
\end{equation*}
By the triangle inequality,
\begin{equation*}
\begin{split}
|\partial_\rho^2 \Psi| &\geq \alpha(\alpha-1)N^\alpha\rho^{\alpha-2} - |\partial_\rho^2(\Phi_G(\phi_+)\rho)|\frac{r}{\tau}\\
&\gtrsim (\alpha-1)N^\alpha -N^2\frac{r}{\tau}\simeq  (\alpha-1 - N^2)N^\alpha\\
&\gtrsim (\alpha-1)N^\alpha,
\end{split}
\end{equation*}
where the last inequality follows from $N \leq N_\alpha$. By the Van der Corput Lemma \cite[Chapter VIII]{stein1993harmonic},
\begin{equation}\label{vdc3}
|II| \lesssim (\alpha-1)^{-\frac{1}{2}}N^{2-\alpha}\tau^{-1}\Big(\| a \|_{L^\infty([\frac{\pi}{4},2\pi])}+\| \partial_\rho a\|_{L^1([\frac{\pi}{4},2\pi])}\Big),  
\end{equation}
since $\frac{r}{\tau}\simeq N^\alpha$. By \eqref{phase3}, $a \in L^\infty([\frac{\pi}{4},2\pi])$ uniformly in $N$. To estimate $\partial_\rho a$, the term that needs most care is $\partial_\rho |\partial_{\phi}^2 \Phi_G(\phi_+)|^{-\frac{1}{2}}$. Since $\phi_+,\theta \in [0,\frac{\pi}{2}]$, we have $\partial_{\phi}^2 \Phi_G(\phi_+) \leq 0$. By \eqref{phase3}, \eqref{phip}, the chain rule
\begin{equation*}
\partial_\rho\Big(\partial_{\phi}^2 \Phi_G(\phi_+)\Big) = \partial_{\rho\phi\phi}\Phi_G(\phi_+) + \partial_{\phi}^3 \Phi_G(\phi_+) \cdot \partial_\rho \phi_+,
\end{equation*}
and the uniform bound
\begin{equation*}
    \sup_{N \leq N_\alpha}|\partial_{\rho}^{k_1}\partial_{\phi}^{k_2}\Phi_G| \lesssim_{k_1,k_2} 1,
\end{equation*}
we have $\partial_\rho a \in L^\infty([\frac{\pi}{4},2\pi])$ uniformly in $N$.

Lastly we show \eqref{lower bound}. Let $C_3(\alpha)>0$ satisfy
\begin{equation*}
C_3(\alpha) = \inf\{C>0:\sup_{v\in \mathbb{R}^2}|J_{\Phi_v,\eta(\frac{\cdot}{N})}| \leq C N^{2-\frac{3}{4}\alpha}\tau^{-\frac{3}{4}},\:\forall \tau>0, N \in S_3\},    
\end{equation*}
and define $C_i(\alpha)$ similarly for $i=1,2$. By \eqref{vdc3}, \eqref{dispersive est3}, we have $\max\limits_{1 \leq i \leq 3}C_i(\alpha)<\infty$. For $\sigma_0 \in \{\frac{3}{4},\frac{5}{6},1\}$ and $\xi \in supp(\eta(\frac{\cdot}{N}))$, we have
\begin{equation}\label{lower bound2}
\lim_{\tau\rightarrow \infty}|J_{\Phi_{v_\xi},\eta(\frac{\cdot}{N})}|\tau^{\sigma_0} \leq C_i(\alpha). 
\end{equation}
The limit above is a constant multiple of the non-zero leading terms given by \eqref{asymptotics} due to the set of critical points of $\Phi_{v_{\xi}}$ whose cardinality is uniformly bounded above for all $\alpha \in (1,2)$ by observing \eqref{group velocity}.

For $i=3,\: \sigma_0 = \frac{3}{4}$, the non-zero contributions to the limit are due to the cusps in $K_3$. Let $N \in S_3$. By \Cref{cusp},
\begin{equation}\label{lower bound3}
    c\cdot \alpha^{-\frac{3}{4}} (2-\alpha)^{-\frac{1}{4}} \leq C_3(\alpha),
\end{equation}
where $c>0$ depends only on $\eta$.

For $i=2,\: \sigma_0 = \frac{5}{6}$, let $N\in S_2$. For $\alpha$ sufficiently close to $2$, we have $N=2^{-2}$. Since $\eta(\frac{\xi(\alpha)}{2^{-2}})\xrightarrow[\alpha\rightarrow 2-]{} 0$ for $\xi = (0,r_\alpha)$, we may replace $\eta(\frac{\cdot}{2^{-2}})$ by another smooth bump function $\Tilde{\eta}$ supported in $\{|\xi| \in [\frac{\pi}{8},\frac{\pi}{2}+\epsilon_0]\}$ where $\epsilon_0>0$ is sufficiently small so that $supp (\Tilde{\eta})\cap K_3 = \emptyset$. Arguing as \eqref{lower bound3} by using \eqref{k2}, one obtains
\begin{equation*}
    (\alpha-1)^{\frac{2}{3}-\frac{5\alpha}{12}} \lesssim_{\Tilde{\eta}} C_2(\alpha).
\end{equation*}
On the contrary, suppose $\alpha>1$ is not close to $2$ such that $N\in S_2$ satisfies $N<2^{-2}$. Then there exists $N^{(\alpha)}\in S_2$ such that $\frac{2}{3\pi}r_\alpha \leq N^{(\alpha)} \leq \frac{4}{3\pi}r_\alpha$. Then $|\eta(\frac{r_\alpha}{N^{(\alpha)}})|\geq c>0$ where $c$ is independent of $\alpha$. Using the same example \eqref{k2}, one obtains
\begin{equation*}
    (\alpha-1)^{\frac{2}{3}-\frac{5\alpha}{12}} \lesssim_{\eta} C_2(\alpha).
\end{equation*}

Lastly, let $N \in S_1$ and $\sigma_0 = 1$. Pick $\xi\in\mathbb{T}^2$ such that $|\xi| = N\pi$. Arguing as \eqref{lower bound2} and invoking \Cref{stationary}, we have
\begin{equation*}
C_1(\alpha) \geq \sqrt{2\pi} |\eta(\frac{\xi}{N})|\cdot |H(\xi,\alpha)|^{-\frac{1}{2}}N^{\alpha-2}\gtrsim (\alpha-1)^{-\frac{1}{2}},
\end{equation*}
where the last inequality follows from using the small angle approximation (see \eqref{sma}, \eqref{sma2}) to obtain
\begin{equation*}
    |H(\xi,\alpha)|\simeq (\alpha-1)N^{-4+2\alpha}.
\end{equation*}
\section{Conclusion and future work.}\label{conclusion}
We have shown, with a convergence rate, the continuum limit of DNLSE on $h\mathbb{Z}^2$ to the FNLSE on $\mathbb{R}^2$ as $h\rightarrow 0$ in the energy subcritical regime for finite time. Our proof employs sharp dispersive estimates that are obtained by studying appropriate degenerate oscillatory integrals. It is of interest to compare the sharp decay rate of $\sigma_0=\frac{3}{4}$ to that in the discrete classical Schr\"odinger equation ($\sigma_0 = \frac{2}{3}$) and the discrete wave equation ($\sigma_0 = \frac{2}{3}$) at the cost of the best constants blowing up as $\alpha\rightarrow 1+,\:2-$. As for future work, it is of interest to extend to the case of mixed fractional derivatives \cite{choi2022well} where (3.2) in dimension two, is replaced by an appropriate discrete analog of

\begin{equation*}
    \Big(-\frac{\partial^2}{\partial x_1^2}\Big)^{\frac{\alpha_1}{2}}
    +\Big(-\frac{\partial^2}{\partial x_2^2}\Big)^{\frac{\alpha_2}{2}}. 
\end{equation*}

By numerical and asymptotic  techniques, we will explore the conditions of highly-localized states in the discrete models that may relate to finite-time blow-up solutions in the continuum limit.
\end{proof}

\section{Acknowledgment.}
Both authors are supported by the U.S. National Science Foundation under the grant DMS-1909559. B. Choi was also supported by an NSF/RTG postdoctoral fellowship under the under the RTG grant DMS-1840260.

\section{Data Availability Statement.}

Data sharing not applicable to this article as no datasets were generated or analysed during the current study.

\section{Competing Interests.}

There are no competing interests other than the National Science Foundation fundings written in the Acknowledgment section.

\bibliographystyle{abbrv}
\bibliography{ref}
\end{document}